\documentclass[12]{article}
\usepackage[all]{xy}
\usepackage{amsmath, amsthm}
\usepackage{amssymb}
\usepackage[mathscr]{eucal}
\usepackage[applemac]{inputenc}
\usepackage{graphicx}
\usepackage{hyperref}

\newcommand{\N}{\mathbb{N}}
\newcommand{\CC}{\mathbb{C}}
\newcommand{\R}{\mathbb{R}}

\newcommand{\Z}{\mathbb{Z}}

\newcommand{\e}{\mathcal{E}}
\def\R{\mathbb{R}}
\def \N {\mathbb{N}}

\def \Fr {\mathrm{Fr}}

\def \cl {\mathrm{Cl}}

\def \id {\mathrm{id}}

\def \pr {\mbox{pr\,}}

\def \inc {\mbox{in\,}}
\def \incl{\mbox{in}}

\newtheorem{lemma}{ \kern \parindent Lemma}[section]

\newtheorem{definition}{ \kern \parindent Definition}[section]
\newtheorem{proposition}{ \kern \parindent Proposition}[section]
\newtheorem{corollary}{ \kern \parindent Corollary}[section]
\newtheorem{theorem}{ \kern \parindent Theorem}[section]
\newtheorem{remark}{ \kern \parindent Remark}[section]
\newtheorem{example}{ \kern \parindent Example}[section]

\def\id{\mathrm{id}}

\def\d{{d}}

\def\e{\mathbf{c}}
\def\l{\mathbf{l}}
\def\r{\mathbf{r}}

\def\t{\mathbf{t}}
\def\d{\mathbf{d}}
\def\F{\check{\pi}_0}

\def\C{\check{C}}
\def\bC{\check{\bar{C}}}
\def\Co{\check{C}_{0}}
\def\bCo{\check{\bar{C}}_{0}}

\title{
A completion construction for  continuous dynamical  systems 
\footnote
{The authors acknowledge the financial support provided by Ministerio de Educación y Ciencia, grant MTM2009-12081, FEDER  and  the University of La Rioja, EGI11/55.}
}
\author{J.M. García Calcines\\
L. J. Hern\'andez Paricio \\
      M. T. Rivas Rodríguez\\
      Departamento de Matem\'atica Fundamental,\\
Universidad de La Laguna, 38271 La Laguna, Tenerife, Spain.\\
Departamento de Matem\'aticas y
Computaci\'on, \\
 Universidad de La Rioja,
c/Luis de Ulloa s/n, 26004 Logro\~no. Spain\\
}

\begin{document}

\maketitle

{\bf Abstract}  In this work  we construct the
$\Co^{\r}$-completion and $\Co^{\l}$-completion of a dynamical
system. If $X$ is a flow, we construct canonical maps $X\to
\Co^{\r}(X)$ and $X\to \Co^{\l}(X)$ and when these maps are
homeomorphism we have the class of $\Co^{\r}$-complete and
$\Co^{\l}$-complete flows, respectively. In this study we  find
out  many relations between the topological properties of the
completions and the dynamical properties of a given flow. In the
case of a complete flow this gives  interesting relations between
the topological properties (separability properties, compactness,
convergence of nets, etc.) and dynamical properties (periodic
points, omega limits, attractors, repulsors, etc.).


\section{Introduction}

Some  origins  of dynamical system and flow theory  can be established with
the pioneering work of H. Poincar{\'e} \cite{Poincare,Poincare2}
in the late XIX century studied the topological properties of
solutions of autonomous ordinary differential equations. We can
also mention the work of A. M. Lyapunov \cite{Liapunov} who
developed his theory of stability of a motion (solution) of a
system of $n$ first order ordinary differential equations. While
much of Poincar{\'e}'s work analyzed the global properties of the
system, Lyapunov's work looks at the local stability of a
dynamical system. The theory of dynamical systems reached a great
development  with the work of G.D. Birkhoff  \cite{Birkhoff}, who
may be considered as the founder of this theory. He established
two main lines in the study of dynamical systems: the topological
theory and the ergodic theory.

This paper develops some new ideas in the topological theory of
dynamical systems based in the theory of exterior spaces. A new
construction ``the $\Co$-completion of a flow" has been introduced
in this work. On the one hand, using this construction we can
clarify the interrelations between some topological properties of
the $\Co$-completion of a flow (separation, local compactness,
local path-connectness, etc) and dymamical properties of the flow
(relations between critical points, periodic points, omega limits,
attractors, etc). On the other hand, many known  results of the
class of $T_2$ compact flows in which all periodic points are
critical can be applied to the completion of a flow $X$ to obtain
nice dynamical properties of the original flow $X$.

Previously, the authors have developed some results on proper
homotopy theory and exterior homotopy theory to classify non
compact spaces and to study the shape of a compact metric space,
see \cite{EHR05,GGH98,GGH04}. Nevertheless, in this  work, our
main objective  is to construct a completion of a flow using
exterior spaces.

Firstly, for an exterior space $X$ we can  construct the limit
space  $L(X)$ and the end  space $\F(X)$. Next, the limit
space and the end space are used to construct the completion
$\Co(X)$. This construction can be considered as a generalization
of the Freudenthal compactification given in \cite{Freudenthal}.

Secondly, we introduce a hybrid structure called exterior flow
that mixes the notion of dynamical system and exterior space. In
particular, we can consider the limit space, the end  space
and the completion of an exterior flow. In the approach given in
this paper, the main key to establish a connection from dynamical
systems to exterior flows is the notion of absorbing open region
(or $\r$-exterior subset). Given a flow on space $X$, an open set
$E$  is said to be $\r$-exterior if for every $x \in X$ there is
$r_0 \in \R$ such that $r\cdot x \in E$ for $r\geq r_0$. The space
$X$ together with the family of $\r$-exterior open subsets is an
exterior flow, which is denoted by $X^{\r}$.

Finally, we can give the limit space, the end space and the completion
of a flow $X$ by applying these constructions, already developed for exterior spaces,
to the exterior flow $X^{\r}$.

An exterior space $X$ is said $\Co$-complete if the natural
transformation $X\to \Co(X)$ is an isomorphism.  For a
$\Co$-complete exterior  space $X$ the limit space $L(X)$ is isomorphic to the end  space
$\F(X)$ and
there is a class of nets
($\pi_0$-$\varepsilon(X)$-nets) which has at least
 a limit point
in $L(X)$.
Moreover, the
limit space $L(X)$ can be considered as a global ``weak
attractor'' of the exterior space $X$.

In section \ref{completionexsp} we have constructed and studied
the completion of an exterior space. We have given some
characterizations, see Theorems \ref{bijectioncero},
\ref{lpcomplete}, of $\Co$-completeness for general exterior
spaces and for locally path-connected exterior spaces. We have
also proved that if $X$ is a locally path-connected exterior space
which is first countable at infinity, then  $\Co(X)$ is
$\Co$-complete.

An interesting class of Hausdorff compact $\Co$-complete spaces is
given in Theorem \ref{completeex}. We have proved that if $X$ is a
Hausdorff compact space and $D$ is a closed totally disconnected
subspace (a Stone space), the externology $\varepsilon(X)$ of open
neighborhoods of $D$ determines a $\Co$-complete exterior space.

In Theorem \ref{condHausdorff} we have given conditions to ensure
that the limit space $L(\Co(X))$ and the completion $\Co(X)$ are
Hausdorff spaces. Under the conditions given in Theorem
\ref{condcompact}, one has that $L(\Co(X))$ and $\Co(X)$ are
compact exterior spaces.

We have proved the following result, which is a generalization of
the Freudenthal compactification:

{\bf Theorem} \ref{main} Suppose that $X$ is a locally
path-connected, connected, Hausdorff exterior space such that the
complement of a exterior open subset is compact. If for every
$x\in X\setminus L(X)$ there is a closed neighborhood $F$ such
that $X\setminus F$ is exterior, then $\Co(X)$ is a Hausdorff
compact exterior space and $L(X)$ is a closed subspace.

The completions construction developed for exterior spaces can be
applied to $\r$-exterior flows. In section
\ref{completionexterior} we have proved that the $\Co$-completion
of an $\r$-exterior flow has a canonical induced
 structure of $\r$-exterior flow,
 moreover, this $\Co$-completion has as limit flow its set of critical points which also is a minimal weak attractor.

As a consequence of these facts, the class of $\Co$-complete
$\r$-exterior flows has very interesting properties: The limit
flow agrees with the end flow which is the set of critical points
and this limit is a  global weak attractor.  When, in addition, the
$\r$-exterior flow is locally path-connected, this limit subflow
has the topology of a Hausdorff totally disconnected space. If a
$\Co$-complete $\r$-exterior flow is also $T_2$, we have that the
limit space agrees with the omega limit of all the flow which,
under these conditions, is the unique minimal global attractor. We
remark that for the class of Hausdorff compact $\Co$-complete
$\r$-exterior flows we have the additional property that the limit
has the topology of a profinite space.

In section \ref{completion}, we consider two full embeddings of
the category of flows to the category of $\r$-exterior flows
denoted by $X\to X^{\r}$ and $X\to X^{\l}$. This permits to
construct the $\Co^{\r}$-completion of a flow and the
$\Co^{\l}$-completion and the corresponding classes of
$\Co^{\r}$-complete and $\Co^{\l}$-complete flows. These classes
of flows have nice properties; for instance, the limit of  a
Hausdorff compact $\Co^{\r}$-complete ($\Co^{\l}$-complete) flow is
the closed subset of critical points which agree with the set of
periodic points and it is the minimal global attractor (repulsor). Moreover, in this case, this limit is a
Stone space. For example, for compact metrizable flows, we have
proved the following result:

 {\bf Theorem} 6.6 Let  $X$ be a locally  path-connected,  compact metric  flow. Then,
$X$ is a  $\Co^{\r}$-complete flow if and only if
  $C(X) =\overline{\Omega^{\r}(X)}$ and $C(X)$ is a Stone space .

Here  $C(X)$ is the set of critical points and $\Omega^{\r}(X)$ is the global $\omega^{\r}$-limit of $X$.
Notation and definitions can be seen in section \ref{prelidin}.

Note that for a flow $X$ with non-empty $\omega^{\r}$-limits,
$\overline{\Omega^{\r}(X)}$ is the minimal closed global
attractor. Therefore, in the context of Theorem 6.6, the set of
critical points of the flow is a minimal closed global attractor.

We point out that a slight  modification of the externologies
considered  in this paper taking as new ${\bf R}$-exterior open
subsets those open subsets $E$ such that for every $x\in X$, there
are $r\in \R$ and $U$ an open neighbourhood at $x$ with
$(r,\infty)\cdot U \subset E$, will permit to introduce new limit
flows, end flows and completions. The techniques developed
in this paper together with the new constructions will allow us to
analyze stability properties of global attractors.

We remark that one of the tools used to study the topology of
(stable) attractors is the shape theory and their associated
algebraic invariants, see for instance \cite{Gabites},
\cite{MR},\cite{MSS},\cite{Sanjurjo03},\cite{Sanjurjo11}. We
emphasize  the fact that for flows on ANR spaces, the
externologies considered in this paper are resolutions of the
limit in the sense of shape theory. Therefore, many of the results
obtained in  all recent developments  on the study of the
properties of (un)stable attractors by means of shape theory can
be related to other analogues given via the theory of exterior
spaces.

\section{Preliminaires on exterior spaces and dynamical systems}

\subsection{The category of  proper and exterior spaces}\label{exteriorspaces}

A continuous map $f:X\rightarrow Y$ is said to be proper if for
every closed compact subset $K$ of $Y$, $f^{-1}(K)$ is a compact
subset of $X$. The category of topological  spaces and the
subcategory of spaces and proper maps will be denoted by {\bf Top}
and {\bf P}, respectively. This last category and its
corresponding proper homotopy category are very useful for the
study of non compact spaces. Nevertheless, one has the problem
that ${\bf P}$ does not have enough limits and colimits and then
we can not develop the usual homotopy constructions like loops,
homotopy limits and colimits, et cetera.
An answer to this problem
is given by  the notion of exterior space. The new  category of
exterior spaces and maps is complete and cocomplete and contains
as a full subcategory the category of spaces and proper maps, see
\cite{GGH98, GGH04}. For more properties and applications of
exterior and proper homotopy categories we refer the reader to \cite{GGH01,
EHR05, DHR09, GHR10, H95}~ and for  a survey of proper homotopy to
\cite{P95}~.

\begin{definition}
Let $(X,\t )$ be a topological space, where $X$ is the subjacent
set and $\t$ its topology. An \emph{externology} on $(X,\t )$ is a
non empty collection $\varepsilon $ (also denoted by $\varepsilon
(X)$) of open subsets which is closed under finite intersections
and such that if $E\in \varepsilon $, $U\in \t$ and $E\subset U$
then $U\in \varepsilon .$ The members of $\varepsilon $ are called
exterior open subsets.
 A subfamily $\mathcal{B}\subset \varepsilon$ is said
to be a base for $\varepsilon$ if for every $E\in \varepsilon$ there
is $B\in \mathcal{B}$ such that $B\subset E$.

 An \emph{exterior space} $(X,\varepsilon,
\t )$ consists of a space $(X,\t )$ together with an externology
$\varepsilon $.

 A map $f:(X,\varepsilon, \t)\rightarrow
(X',\varepsilon', \t')$ is said to be an \emph{exterior map} if it
is continuous and $f^{-1}(E)\in \varepsilon $, for all $E\in
\varepsilon '.$
\end{definition}

The category of exterior spaces and exterior maps will be denoted by {\bf
E}. Given a space $(X, \t_{X})$, we can always consider the
trivial exterior space taking $\varepsilon = \{ X \}$ or the total
exterior space if one takes $ \varepsilon = \t_{X}$. In this paper,
we shall use the exterior space of real numbers $(\R,\r)$, where $\r$ is the externology
determined by the externology base $\{(n,+\infty)| n \in \Z\}$.
An important example of externology on a given topological space
$X$ is the one constituted by the complements of all
closed-compact subsets of $X$, that will be called the cocompact
externology and usually written as $\varepsilon ^{\e}(X).$ The
category of spaces and proper maps can be considered as a full
subcategory  of the category of exterior spaces via the full
embedding $(\cdot)^{\e}:{\bf P}\hookrightarrow {\bf E}.$ The
functor $(\cdot)^{\e}$ carries a space $X$ to the exterior space
$X^{\e}$ which is provided with the topology of $X$ and the
externology $\varepsilon ^{\e}(X)$. A map $f\colon X \to Y$ is
carried to the exterior map $f^{\e}\colon X^{\e} \to Y^{\e}$ given
by $f^{\e} = f$. It is easy to check that a continuous map
$f:X\rightarrow Y $  is proper if and only if $f=f^{\e}\colon
X^{\e} \to Y^{\e}$  is exterior.

An important role in this paper will be played by the following
construction $(\cdot)\bar \times (\cdot) $: Let $(X, \varepsilon
(X),\t_X)$ be an exterior space, $(Y,\t_Y)$ a topological space
and for $y\in Y$ we denote by $(\t_{Y})_y$ the family of open
neighborhoods of $Y$ at $y$. We consider on $X\times Y$ the
product topology $\t_{X \times Y}$ and the externology
$\varepsilon ({X\bar \times Y})$ given by those $E\in \t_{X\times
Y}$ such that for each $y\in Y$ there exists
 $U_y\in (\t_{Y})_y$ and $T^{y}\in \varepsilon ({X})$ such that
$T^{y}\times U_{y}\subset E.$ This exterior space will be denoted by
$X\bar{\times }Y$ in order to avoid a possible confusion with the natural product
externology. This construction gives a functor
$$(\cdot)\bar \times (\cdot) \colon {\bf E}\times {\bf Top}\rightarrow
{\bf E}.$$
When $Y$ is a compact space,  we have that $E$ is an exterior open
subset of $X\bar \times Y$ if and only if it is an open subset and there exists $G\in
\varepsilon ({X})$ such that $G\times Y\subset E$. Furthermore, if
$Y$ is a compact space and $\varepsilon (X) =\varepsilon
^{\e}(X),$ then $\varepsilon (X\bar{\times } Y)$ coincides with
$\varepsilon^{\e}(X\times Y)$ the externology of the complements of
closed-compact subsets of $X\times Y.$ We also note that if $Y$ is
a discrete space, then $E$ is an exterior open subset of $X\bar \times Y$ if and only
if it is open  and for each $y\in Y$ there is $T^y\in \varepsilon ({X})$ such that $T^y\times \{y\}\subset E$.

This bar construction provides a natural way to define {\em
exterior homotopy} in ${\bf E}$. Indeed, if $I=[0,1]$ denotes the closed
unit interval, given exterior maps $f,g:X\rightarrow Y,$ it is
said that $f$ \emph{is exterior homotopic to} $g$ if there exists
an exterior map $H:X\bar{\times }I\rightarrow Y$ (called exterior
homotopy) such that $H(x,0)=f(x)$ and $H(x,1)=g(x),$ for all $x\in
X.$ The corresponding homotopy category of exterior spaces will be
denoted by $\pi {\bf E}.$ Similarly, the usual homotopy category
of topological spaces will be denoted by $\pi {\bf Top}.$

\subsection{Dynamical Systems and $\Omega$-Limits}\label{prelidin}

Next we recall some elementary concepts about dynamical systems;
for more complete descriptions and properties
we refer the reader to \cite{Bhatia}.

\begin{definition}
A \emph{dynamical system} (or a \emph{flow}) on a topological  space
$X$ is a continuous map $\varphi \colon \R {\times} X\to  X$,
$\varphi(t,p)=t \cdot p$,  such that
\begin{enumerate}
\item[(i)] $\varphi(0,p)=p,$ $\forall p \in X;$
\item[(ii)] $\varphi(t,\varphi(s, p))=\varphi( t+s,p),$ $\forall p
\in X,$ $\forall t,s \in \mathbb{R}.$ \end{enumerate} A flow on
$X$ will be denoted by $(X, \varphi)$ and when no confusion is
possible, we use $X$ for short.
\end{definition}

For a subset $A \subset X$, we denote inv$(A)=\{p \in A| \R \cdot p \subset A\}$~.

\begin{definition}
A subset $S$ of a flow $X$ is said to be \emph{invariant} if
inv$(S)=S.$
\end{definition}

Given a flow $\varphi \colon \R {\times} X\to  X$ one has a
subgroup $\{\varphi_t \colon X \to X| t\in \R\}$ of
homeomorphisms, $\varphi_t(p) = \varphi(t,p)$, and a family of
motions  $\{\varphi^p\colon \R \to X| p \in X\}$, $\varphi^p(t) =
\varphi(t,p).$

\begin{definition} Given two flows  $\varphi \colon \R {\times} X\to  X$,
$\psi \colon \R {\times} Y\to  Y$, a  \emph{flow morphism}
$f\colon (X, \varphi) \to (Y, \psi)$ is a continuous map $f\colon
X \to Y$ such that $f(r\cdot p)=r\cdot f(p)$ for every $r\in \R$
and for every $p\in X$.
\end{definition}

We note that if $S \subset X$ is invariant, $S$ has a flow
structure and the inclusion is a flow morphism.

We denote by $\bf{
F}$ the category of flows  and flows morphisms.

\medskip

We recall some basic fundamental examples: (1) $X=\R$ with the action $\varphi \colon \R \times X \to X$,  $\varphi(r, s)= r+s$. (2) $X=S^1=\{z\in {\mathbb{C}}| |z|=1\}$ with  $\varphi \colon \R \times
X \to X$, $\varphi(r, z)= e^{2\pi i r}z$. (3) $X=\{0\}$ with the trivial
action $\varphi \colon \R \times X \to X$ given by
$\varphi(r, 0)= 0.$ In all these cases, the flows have only one trajectory.

\begin{definition} For a flow $X$,
the \emph{$\omega^{\r}$-limit set} (or right-limit set, or positive
limit set) of a point $p \in X$ is given as follows:
$$\omega^{\r} (p) =\{q \in X | \exists \mbox{ a net }  t_{\delta} \to +\infty \mbox{ such that }
t_{\delta}\cdot p \to q\}.$$
\end{definition}

 If $\overline{A}$ denotes the closure of a subset $A$ of a topological space,
we  note that the subset $\omega^{\r}(p)$ admits the alternative
definition:
$$\omega^{\r} (p) = \bigcap_{t \geq 0}\overline{ [t, +\infty) \cdot p}$$
which has the advantage of showing that $\omega^{\r}(p)$ is
closed.

\begin{definition}
The \emph{$\Omega^{\r}$-limit set} of a flow $X$  is the following
invariant subset:
$$\Omega^{\r} (X) =\bigcup_{p\in X} \omega^{\r}(p)$$
\end{definition}

Now we introduce the basic notions of critical, periodic and
$\r$-Poisson stable points.

\begin{definition} Let $X$ be a flow.
A point $x\in X$ is said to be  a \emph{critical point} (or a
\emph{rest point}, or an \emph{equilibrium point})  if for every
$r\in \R,$ $r\cdot x =x.$ We denote by $C(X)$ the invariant subset
of critical points of $X.$
\end{definition}

\begin{definition} Let $X$ be a flow. A point $x\in X$ is said to be
\emph{periodic} if there is $r\in \R$, $r\not = 0$ such that
$r\cdot x= x.$ We denote by $P(X)$ the invariant subset of
periodic points of $X.$
\end{definition}

It is clear that a critical point is a periodic point. Then $$C(X)
\subset P(X).$$

If $x \in X$ is a periodic point but not critical, then there exists a real number
$r\not =0$ such that $r\cdot x=x$ and $r$ is called a
\emph{period} of $x.$ The smallest positive period $r_0$ of $x$ is
called the {\it fundamental period} of $x.$ Further if $r \in \R$
is such that $r\cdot x=x,$ then there is $z\in \Z$ such that
$r=zr_0.$

 \begin{definition} Let $(X, \varphi)$ be a flow. A point $x \in X$
is said to be $\r$-\emph{Poison stable} if there is a net
$t_{\delta} \to +\infty$ such that $t_{\delta} \cdot x \to x;$ that is, $x\in
\omega^{\r}(x).$ We will denote by $P^{\r}(X)$ the invariant
subset of $\r$-Poison stable points of $X.$
\end{definition}

The reader can easily check that $$P(X)\subset P^{\r}(X) \subset
\Omega^{\r}(X).$$

\begin{definition}\label{attractoruno}  Let $(X, \varphi)$ be a flow and $M\subset X$
be an invariant subspace. The subset
$A_{\omega^{\r}}(M)=\{x \in X| \omega^{\r}(x) \cap M\not = \emptyset\} $
is called the \emph{region of weak attraction} of $M$. If $A_{\omega^{\r}}(M)$
is a neighbourhood of $M$, $M$ is said to be a \emph{weak attractor}, and  if
we also have that $A_{\omega^{\r}}(M)=X$, $M$ is said to be a \emph{global weak attractor}.
 The subset
$A_{{\r}}(M)=\{x \in X| \omega^{\r}(x)  \not = \emptyset,  \omega^{\r}(x) \subset M \}$
is called the \emph{region of  attraction} of $M$. If $A_{{\r}}(M)$
is a neighbourhood of $M$, $M$ is said to be an \emph{attractor}, and if
we also have that $A_{{\r}}(M)=X$, $M$ is said to be a \emph{global attractor}.

\end{definition}

 The notions above can be dualized to obtain the notion of the
$\omega^{\l}$-limit ($\l$ for `left') set of a point $p$, the
$\Omega^{\l}$-{ limit} of  $X$, $\l$-Poison stable points, repulsors et
cetera.

\section{The completion of an exterior space}\label{completion}\label{completionexsp}

\subsection{End Space and Limit Space of an exterior space}\label{tres}
In this section we will deal with special limits for a given
exterior space. We assume that the reader is familiarized with
this particular categorical construction. However, for a
definition and main properties of limits (of inverse systems of
spaces) we refer the reader to \cite{Engelking}. A more detailed
and complete study of some notions given in this subsection can
be seen in \cite{GHR12}.

Given an exterior space $X=(X, \varepsilon(X))$, its externology
$\varepsilon(X)$ can be seen as an inverse system of spaces, then
we  define  the limit space of $(X, \varepsilon(X))$ as the
topological space $$L(X)=\lim  \varepsilon(X).$$

Note that for each $E' \in \varepsilon(X)$  the canonical map
$\lim  \varepsilon(X) \to E' $ is continuous and factorizes as
$\lim  \varepsilon(X) \to \cap_{E\in \varepsilon(X)}E\to  E'.$
Therefore the canonical map $\lim  \varepsilon(X) \to \cap_{E\in
\varepsilon(X)}E$ is continuous. On the other side, by the
universal property of the inverse system, the family of maps $
\cap_{E\in \varepsilon(X)}E\to E',$ $E' \in \varepsilon(X)$
induces a continuous map  $\cap_{E\in \varepsilon(X)}E\to  \lim
\varepsilon(X)$~. This implies that the canonical map $\lim
\varepsilon(X) \to \cap_{E\in \varepsilon(X)}E$ is a
homeomorphism.

We recall that for a topological space $Y$, $\pi_0(Y)$ denotes the
set of path-components of $Y$~ and we have a  canonical map $q_0\colon Y\to
\pi_0(Y)$~ which induces a quotient topology on $\pi_0(Y).$
We remark
that if $Y$ is locally path-connected
then $\pi_0(Y)$
is a discrete space.

\begin{definition} Given an exterior space $X=(X, \varepsilon(X))$
the \emph{limit space} of $X$ is the topological subspace
$$L(X)=\lim \varepsilon(X)= \cap_{E\in \varepsilon(X)}E.$$

The \emph{end space} of $X$ is the inverse limit
$$\F(X)=\lim \pi_0 \varepsilon(X)=\lim_{E\in \varepsilon(X)}\pi_0(E)$$
provided with the inverse limit topology of the spaces $\pi_0(E).$

\end{definition}


Note that an end point $a\in \F(X)$ can be represented by  $a=(C_E)_{E\in \varepsilon(X)}$,
where $C_E$ is a path-component of $E$ and if $E\subset E'$, $C_E\subset C_{E'}$.
In this paper the canonical maps associated to the limit inverse construction will be denoted by
$\eta_0\colon \F(X) \to \pi_0(E)$ (if necessary by $\eta_{0,E}$) given by $\eta_0((C_{E'})_{E'\in \varepsilon(X)})=C_E$.
If $E\subset E'$, we denote by $\eta_{E'}^E \colon \pi_0(E)\to  \pi_0(E')$ the natural induced map.

We note that our notion of end point for exterior spaces described by using path-components generalizes
the notion of ideal point introduced by Ker\'ekj\'art\'o \cite{Kererjarto} for surfaces and by Freudenthal \cite{Freudenthal}
for more general spaces.

\begin{definition} Given an exterior space $X=(X, \varepsilon(X))$,
the \emph{bar-limit space} of $X$ is the topological subspace
$$\bar L(X)=\lim_{E\in \varepsilon(X)}\overline{E}= \cap_{E\in \varepsilon(X)}\overline{E}.$$

\end{definition}

\begin{remark} We note that if an externology $\varepsilon(X)$ satisfies the additional condition that
the $\emptyset \not \in \varepsilon(X)$, then  $\varepsilon(X)$ is a filter of open subsets. The usual
notion of the set of cluster points of a filter can also be considered for externologies. In this way,
we can say that $\bar L(X)$ is the set of cluster points of the externolgy $\varepsilon(X)$.

We note that $ L(X),{\bar L}(X)$ also have  natural structures of exterior space with the relative externologies.
The relative externology on $L(X)$ is trivial, but this fact is not necessarily true  on ${\bar L}(X)$.
\end{remark}

It is interesting to observe that if $X$ is an exterior space and
$X$ is locally path-connected,
$\F(X)$
is a prodiscrete space.
On the other hand,
given any exterior space $(X, \varepsilon(X))$, we have a
canonical continuous map
$$e_0 \colon L(X) \to \F(X)$$
and a canonical inclusion

$$L(X) \to \bar L(X)$$

\begin{definition} Given an exterior space $X=(X, \varepsilon(X))$,
an end point $a\in \F(X)$
is said to be
$e_0$-\emph{representable}
if there is $p\in L(X)$ such that  $e_0(p)=a$.
Notice that the map $e_0\colon L(X)\to \F(X)$
induce an \emph{$ e_0$-decomposition}
$$L(X)= \bigsqcup_{a \in  \F(X)} L_a^0(X)$$
where $L_a^0(X)=
e_0^{-1}(a)$.
This  subset will be
 called the \emph{$e_0$-component} of the end $a\in \F(X)$.
\end{definition}

If we denote by $e_0L(X)$ the subset of
representable end points. It is clear that
$$ L(X)= \bigsqcup_{a \in e_0L(X)} L_a^0(X)$$

\begin{example} \quad Let $M\colon \R \to (0,1)$ be an increasing
continuous map such that $\lim_{t\to -\infty}M(t)=0$ and
$\lim_{t\to +\infty}M(t)=1$ and take $A=\{e^{2\pi i t}|t\in \R\}$,
$B=\{M(t)e^{2\pi i t}|t\in \R\}$. Consider $X=A\cup B\subset \CC$
provided with the relative topology (observe that $X$ is not
locally connected). On the topological space $X$ the  flow
$\varphi \colon \R \times X \to X$ is given by $\varphi (r,e^{2\pi
i t})=e^{2\pi i (r+t)}$, $\varphi (r,M(t) e^{2\pi it})=M(r+t)
e^{2\pi i (r+t)}$. It is clear that this flow has two trajectories
$A,B$. If for each natural number $n$ we denote $B_n=\{M(t)e^{2\pi
i t}|t\geq n\}$,
we can take the externology $\varepsilon(X)=\{E\in \t_X| \exists n \mbox{ such that } A\cup B_n \subset E\}$.
Since $A,B_n$ are
path-connected, it follows
that $\pi_0(A\cup B_n)=\{A,B_n\}$. Therefore, one can check that
 $$\check \pi_0(X)=\{*_A, *_B\}$$
 For this example we have  $L(X)=A$, the $e_0$-decomposition
 $$L_{*_A}^0=A, \quad L_{*_B}^0=\emptyset$$
 This means  that $*_B $ is not $e_0$-representable. Note that
  $\overline{B_n}=E_n$ is connected. This implies that if we take the set of connected components
  of an exterior open subset instead of the set of path-components, the corresponding inverse limit
  will have only one point.
\end{example}

\begin{remark} From the definition of $e_0$-representable end it follows that
$e_0 \colon L(X) \to \F(X)$ is surjective if and only if all end points are $e_0$-representable.
\end{remark}

\begin{proposition}\label{bijective} Let $X$ be an exterior space and consider the natural transformation
$e_0 \colon L(X) \to \F(X)$. Then, the following conditions are equivalent:
\begin{itemize}
\item[(i)] $e_0 \colon L(X) \to \F(X)$ is injective.
\item[(ii)] For every $x, x' \in L(X), x\not = x'$, there is $E\in \varepsilon(X)$ and $C, C'$ path-components of $E$
such that $x\in C$, $x'\in C'$ and $C \cap C' = \emptyset$.
\item[(iii)] For every $x \in L(X)$,  $\bigcap_{E \in \varepsilon(X)} C_E(x)=\{x\}$, where $C_E(x)$ denotes the path-component of
$x$ in $E$
\end{itemize}
\end{proposition}

Given an exterior space $X=(X, \varepsilon(X))$ and  $E\in \varepsilon(X)$,  a subset  $W\subset E$ is said to be $q_0$-saturated if
$q_0^{-1}(q_0(W))=W$; that is,  $W$ is a union of path-components of $E$.

\begin{proposition}\label{bijective} Let $X$ be an exterior space and suppose that  the natural transformation
$e_0 \colon L(X) \to \F(X)$ is a bijection. If for every $x\in L(X)$, $x \in U$,
$U$ open in $X$, there is $E\in \varepsilon (X)$ and $W$ $q_0$-saturated open subset  in $E$ such that
$x\in W \subset U$, then $e_0 \colon L(X) \to \F(X)$ is an open map. Therefore, under these conditions, $e_0 \colon L(X) \to \F(X)$ is an
homeomorphism.
\end{proposition}

\begin{proof} Suppose that $x\in L(X)$ and $x \in U$, where $U$ is open in $X$. By hypothesis,
there is $E\in \varepsilon(X)$ and $W $ $q_0$-saturated open subset  in $E$ such that $x\in W \subset U$.
Then $e_0(x)\in e_0(W\cap L(X))\subset e_0(U\cap L(X))$. Since $W$ is $q_0$-saturated,
$e_0(W\cap L(X))=\eta_0^{-1}(q_0(W))$ is an open subset of $\F(X)$.
\end{proof}

An interesting property of $\F(X)$ is given in the next result, whose proof is contained in
Theorem 3.17 of  \cite{GHR12}. By a locally locally compact at
infinity exterior space we mean an exterior space such that for every $E\in \varepsilon(X)$,
there is  $E'\in \varepsilon(X)$ with $\overline{E'}\subset E$ and $\overline{E'}$ compact.

\begin{proposition}\label{masmasigual} Let $X=(X,\varepsilon(X))$ be an exterior space and
suppose that $X$ is locally path-connected and locally compact at
infinity. Then,
$\F(X)$ is a profinite compact space.\end{proposition}

On the other hand,  if $X,Y$ are exterior spaces, and $f\colon X
\to Y$ is an exterior map, then $f$  induces continuous maps
$L(f)\colon L(X) \to L(Y)$, $\F(f)\colon \F(X) \to \F(Y)$.
It is not difficult to check that
$L$ preserves exterior homotopies and $\F$
is invariant by
exterior homotopy:

\begin{lemma}\label{lemahomotopia} Suppose that $X$ and $Y$ are exterior spaces and
$f,g \colon X\to Y$ exterior maps.
\begin{itemize}
\item[(i)] If $H \colon X\bar \times I \to Y$ is an exterior homotopy from $f$ to $g$,
then $L(H)=H|_{L(X) \times I} \colon  L(X\bar \times I)=L(X)\times
I  \to  L(Y)$ is a homotopy from $L(f)$ to $L(g).$
\item[(ii)] If $f$ is exterior homotopic to $g$, then $\F(f)=\F(g)$.
\end{itemize}
\end{lemma}

As a consequence of this lemma one has:

\begin{proposition}\label{homotopy} The functors $L, \F
\colon {\bf E}
\to {\bf Top}$ induce functors  $$L\colon \pi {\bf E} \to \pi{\bf
Top},\hspace{15pt}\F
 \colon \pi{\bf E} \to {\bf Top}.$$
\end{proposition}

\subsection{The functor $\Co  \colon {\bf E} \to {\bf E}$}\label{exteriorcompletion}

In this section, we develop the main construction of this paper: the completion of an exterior space.
Later, we will apply this technique to construct some completions of flows.

Given an exterior space $X=(X, \varepsilon(X))$, we can take the following push-out square
$$\xymatrix{L(X) \ar[r]^{e_0} \ar[d] & \F(X) \ar[d]^{\inc_0}\\
X \ar[r]^-{p_0} & X\cup_{L(X)}\F(X) }
$$
where $L(X) \to X$ is the canonical inclusion and $p_0 \colon X \to X\cup_{L(X)}\F(X)$ and
${\inc_0}\colon \F(X) \to X\cup_{L(X)}\F(X)$ are the canonical  continuous maps of the push-out.

We can consider the push-out topology: $V \subset X\cup_{L(X)}\F(X)$ is open
if $p_0^{-1}(V)$ is open in $X$ and $\incl_0^{-1}(V)$ is open in $\F(X)$.
%

\begin{example}\label{sum} For  $(\R, \r)$ we have: $L(X)=\emptyset$, the topology of  $ X\cup_{L(X)}\F(X)$ is the disjoint
sum of $\R$ and $\{\infty\}$~.
\end{example}

In order to have a good relation between the neighborhoods of an
end point and its corresponding family of path-components, we
reduce the family of open subsets of the push-out topology to a
new family $\mathcal{G}_0$. Recall that for each $E\in
\varepsilon(X)$ we have the canonical maps:

$$
\xymatrix{  &E\ar[d]^{q_0}\\
\F(X)\ar[r]^{\eta_0}& \pi_0 (E) }
$$

Then, given  $W\subset  X\cup_{L(X)}\F(X)$, $W\in \mathcal{G}_0$
if and only if it satisfies:

 \begin{itemize}

\item[(i)] $p_0^{-1}(W)$ is open in $X$, $\incl_0^{-1}(W)$ is open in $\F(X)$, and

\item[(ii)] for each $a \in \incl_0^{-1}(W)$  there is $E\in \varepsilon(X)$ and there is $G$ an open in $\pi_0(E)$ such that
$a\in \eta_0^{-1}(G)\subset \incl_0^{-1}(W)$,
$q_0^{-1}(G) \subset p_0^{-1}(W)$.

 \end{itemize}

\begin{proposition} Given an exterior space $(X, \varepsilon(X))$, then
 the family $\mathcal{G}_0$  of subsets of $X\cup_{L(X)}\F(X)$ is a topology. Moreover,
  $ \incl_0\colon \F(X) \to \incl_0(\F(X))$ is a homeomorphism.
\end{proposition}

\begin{proof} Consider a family  subsets $W_i \in \mathcal{G}_0$~.
Since $p_0^{-1}(\cup_{i}W_i)=\cup_{i}p_0^{-1}(W_i)$, one has that $p_0^{-1}(\cup_{i}W_i)$ is open in $X$
and in a similar way  $
\incl_0^{-1}(\cup_{i}W_i)$ in open in $\F(X)$.

Suppose that
$a \in \incl_0^{-1}(W_{i_0})$, then  there is $E\in \varepsilon(X)$ and $G$ open in $\pi_0(E)$ such that
$a\in \eta_0^{-1}(G)\subset \incl_0^{-1}(W_{i_0}) \subset \incl_0^{-1}(\cup_{i}W_i)$,
$q_0^{-1}(G) \subset p_0^{-1}(W_{i_0})\subset p_0^{-1}(\cup_{i}W_i)$.
 Therefore $\cup_{i}W_i$ is in $\mathcal{G}_0$.

Now suppose that $W_1, W_2$ are $\mathcal{G}_0$~.
Since $p_0^{-1}(W_1\cap W_2)=p_0^{-1}(W_1) \cap p_0^{-1}(W_1)  $, one has that  $p_0^{-1}(W_1\cap W_2)$ is open in $X$~and similarly $\incl_0^{-1}(W_1\cap W_2)$ is open  in $\F(X)$.
Suppose that $a\in \incl_0^{-1}(W_1\cap W_2)$. Then
there are $E_1, E_2\in \varepsilon(X)$ and $G_1$ open in $\pi_0(E_1)$, $G_2$ open in $\pi_0(E_2)$ such that
$a\in \eta_0^{-1}(G_1)\subset \incl_0^{-1}(W_1)$,
$q_0^{-1}(G_1) \subset p_0^{-1}(W_1)$, $a\in \eta_0^{-1}(G_2)\subset \incl_0^{-1}(W_2)$,
$q_0^{-1}(G_2) \subset p_0^{-1}(W_2)$. If
$E= E_1\cap E_2 \in \varepsilon(X)$, we can consider the continuous maps $\eta_{E_1}^E\colon \pi_0(E)\to \pi_0(E_1)$,
$\eta_{E_2}^E\colon \pi_0(E)\to \pi_0(E_2)$ and  $G=(\eta_{E_1}^E)^{-1}(G_1)\cap (\eta_{E_2}^E)^{-1}(G_2)$~. This $G$ satisfies
$a\in\eta_0^{-1}(G)\subset \incl_0^{-1}(W_1\cap W_2)$,
$q_0^{-1}(G) \subset p_0^{-1}(W_1\cap W_2)$.
Therefore $W_1 \cap W_2$ is in $\mathcal{G}_0$.

We also note that if $G$ open in $\pi_0(E)$, then
$\incl_0(\eta_0^{-1}(G))= (p_0(q_0^{-1}(G))\cup \incl_0(\eta_0^{-1}(G))) \cap \incl_0(\F(X))$.
This implies that  $ \incl_0\colon \F(X) \to \incl_0(\F(X))$ is a homeomorphism.

\end{proof}

We note that the topology $\mathcal{G}_0$ is  coarser than the push-out topology. For instance, one has:

\begin{example} For  $X=(\R, \r)$ ,  $X\cup_{L(X)}\F(X)$ with the topology $\mathcal{G}_0$  is
homeomorphic to $(0, 1]$ with the usual topology.
\end{example}

Compare this example with the example \ref{sum}. With the new coarser topology $\mathcal{G}_0$,
a neighborhood at 1, which corresponds to $\infty$, always contains a representative path-component
of the end point.

\medskip

If $V$ is a $q_0$-saturated open subset in $E$ , $E\in \varepsilon(X)$, denote
$$W_0(V)=p_0(V)\cup\incl_0(\eta_0^{-1}(q_0(V))).$$ It is easy to check that
$p_0^{-1}(W_0(V))= V$ and $\incl_0^{-1}(W_0(V))= \eta_0^{-1}(q_0(V))$, then by construction $W_0(V) \in \Co(X)$.
 As a consequence of this fact, one has:

 \begin{lemma}  If $V$ is a $q_0$-saturated  open subset in $E$, $E\in \varepsilon(X)$, then
$W_0(V)$  is in $\mathcal{G}_0$. In particular, for $V=E$ one has that $W_0(E)=p_0(E)\cup \incl_0(\F(X))$ is  in $\mathcal{G}_0$ and $p_0^{-1}(W_0(E))=E$, $\incl_0^{-1}(W_0(E))= \F(X)$.

\end{lemma}

We consider the family $\{W_0(E)|E\in \varepsilon(X)\}$. Then if
$E_1,E_2\in \varepsilon(X)$, one can check that $W_0(E_1)\cap W_0(E_2)=W_0(E_1\cap E_2)$.
Now if $U$ is  in $\mathcal{G}_0$ and $U\supset W_0(E)$, then $p_0^{-1}(U)\supset p_0^{-1}(W_0(E))=E$.
This implies that $p_0^{-1}(U)\in \varepsilon (X)$ and
we have that $W_0(p_0^{-1}(U))=U$.
As a consequence:

\begin{lemma} The family $\{W_0(E)|E\in \varepsilon(X)\}$  is an externology in the topological space $(X\cup_{L(X)}\F(X), \mathcal{G}_0)$.
\end{lemma}

\begin{definition} The push-out $X\cup_{L(X)}\F(X)$ with the topology $\mathcal{G}_{0}$ and
the externogy $\{W_0(E)|E\in \varepsilon(X)\}$ has the structure of an exterior space that will be called the \emph{$\Co$-completion} of $X$ and it will be denoted by $\Co(X)$.
\end{definition}

In the following examples, we analyze the completion functor for trivial and total externologies.

\begin{example} Suppose that $(X,\{X\})$ is a trivial exterior space. Then $L(X)=X$,
$\F(X)=\pi_0(X)$. Therefore,
$\Co(X)=\pi_0(X)$ (notice that we have the quotient topology and the trivial externology).
\end{example}

\begin{example} Suppose that $X$  is a total exterior space $\varepsilon(X)=\t_X$. Then $L(X)=\emptyset $,
$\F(X)=\emptyset$. Therefore,
$\Co(X)=X$.
\end{example}

From the properties of the push-out construction one can easily check:

\begin{lemma} If $X$ is an exterior space, then $L(\Co(X))\cong \incl_0(\F(X))$.
\end{lemma}

\begin{proposition} \label{functor}The construction $\Co\colon {\bf E} \to {\bf E}$ is a functor and
there is a canonical transformation $p_0 \colon \id_{\bf E}\to \Co$.
\end{proposition}

\begin{proof}

Given an exterior map $f\colon X \to Y$, we have to prove that
 $\Co (f) \colon \Co(X) \to \Co(Y)$ is  exterior.
Suppose that $W$ is open in $\Co(Y)$. Then,
$p_0^{-1}(W)$ is open in $Y$, $\incl_0^{-1}(W)$ is open in $\F(Y)$ and for each
$b \in \incl_0^{-1}(W)$  there is $E\in \varepsilon(Y)$ and $G$ open in $\pi_0(E)$ such that
$b \in  \eta_0^{-1}(G)\subset \incl_0^{-1}(W)$,
$q_0^{-1}(G) \subset p_0^{-1}(W)$.
To see that $(\Co (f))^{-1}(W)$ is open in  $\Co(X)$, we note that
$p_0^{-1}((\Co (f))^{-1}(W))=f^{-1}(p_0^{-1}(W))$ is open in $X$ and
$\incl_0^{-1}((\Co (f))^{-1}(W))=(\F(f))^{-1}(\incl_0^{-1}(W))$ is open in $\F(X)$; moreover, for each
$a \in \incl_0^{-1}((\Co (f))^{-1}(W))$  there is $E\in \varepsilon(Y)$ and $G$ open in $\pi_0(E)$ such that
$\F(f)(a) \in \eta_0^{-1}(G)\subset \incl_0^{-1}(W)$,
$q_0^{-1}(G) \subset p_0^{-1}(W)$. If we take $f^{-1}(E)\in \varepsilon(X)$ and
$(\pi_0(f|_{f^{-1}(E)}))^{-1}(G)$ open in $\pi_0(f^{-1}(E))$, then one has that
$a \in \eta_0^{-1}(\pi_0(f|_{f^{-1}(E)})^{-1}(G))\subset \incl_0^{-1}((\Co (f))^{-1}(W))$ and
$q_0^{-1}((\pi_0(f|_{f^{-1}(E)}))^{-1}(G))\subset p_0^{-1}((\Co (f))^{-1}(W))$. This implies that
$(\Co (f))^{-1}(W)$ is open in  $\Co(X)$. Therefore $\Co (f)$ is a continuous map.

To see that $\Co (f)$ is exterior it suffices to check that $(\Co (f))^{-1}(W_0(E))\supset W_0(f^{-1}(E))$
for each $E\in \varepsilon(X)$.

To see that $p_0$ is a natural transformation, we note that $p_0^{-1}(W_0(E))=E$.

\end{proof}

\begin{definition} An exterior space $X$ is said to be \emph{$\Co$-complete}  if the canonical
map $p_0\colon  X \to \Co(X)$ is an isomorphism in ${\bf E}$~. An exterior space is said to be
\emph{$\Co^2$-complete} if $\Co(X)$ is $\Co$-complete.
\end{definition}

\begin{theorem} The functor $\Co \colon {\bf E}|_{\Co^2-complete}\to   {\bf E}|_{\Co-complete}$ is left adjoint to the inclusion functor  In$\colon {\bf E}|_{\Co-complete}  \to {\bf E}|_{\Co^2-complete}$.
\end{theorem}

\begin{proof} Firstly, we observe that if $X\in {\bf E}$ is $\Co^2$-complete, one has that $\Co(X)$ is $\Co$-complete.
Now take $X$ $\Co^2$-complete and
$Y$ $\Co$-complete. If $f\colon \Co(X) \to Y$ is a map in ${\bf E}$, then we have an
induced map $fp_0^X\colon X \to Y$. And if $g \colon X \to Y$ is a map in ${\bf E}$, then
$(p_0^Y)^{-1}\Co(g)\colon\Co(X) \to \Co(Y) \cong Y$ is a map such that $(p_0^Y)^{-1}\Co(g)p_0 ^X=g$. The adjuntion above is a consequence of this bijective correspondence.
\end{proof}

\begin{theorem}\label{bijectioncero}  An exterior space $X$ is  $\Co$-complete if and only if the canonical map
$L(X)\to \F(X)$ is  bijective and satisfies the following condition: If $x\in L(X)$, $x \in U$,
$U$ open in $X$, there are $E\in \varepsilon (X)$ and $W$ $q_0$-saturated open subset  in $E$ such that
$x\in W \subset U$~.
\end{theorem}

\begin{proof} We are going to prove that $X$ is  $\Co$-complete by using Proposition \ref{bijective}. Firsty,    it  is easy to check  that $p_0\colon X \to \Co(X)$ is bijective and exterior.
Now, if $U$ is open in $X$ and $U\cap L(X)=\emptyset$ it is clear that $p_0(U)$ is open
in  $\Co(X)$. If $U\cap L(X) \not = \emptyset $, then if $a\in \incl_0^{-1}(p_0(U))$, there is
a unique $x\in L(X)$ such that $x\in U$ and $p_0(x)=a$. By  hypothesis conditions, there are $E\in \varepsilon (X)$ and $W$ $q_0$-saturated open subset  in $E$ such that
$x\in W \subset U$~. This implies that $q_0(W)$ verifies that
$\eta_0^{-1}(q_0 (W))\subset \incl_{0}^{-1}p_0(U)$
and $q_0^{-1} q_0(W)=W \subset U=p_0^{-1} p_0 (U)$. Therefore  $p_0(U)$ is open
in  $\Co(X)$ and $p_0$ is an exterior homeomorphism.  We also remark that if  $E \in \varepsilon(X)$, then
$p_{0}(E)=W_{0}(E)$ is also exterior. This implies that $p_{0} \colon X \to \Co(X)$ is an isomorphism in {\bf E}. Thus we have obtained that $X$ is $\Co$-complete. The converse is plain to check.
\end{proof}

\begin{definition} An exterior space $X=(X,\varepsilon(X))$ is said to be
\emph{first countable at infinity} if $\varepsilon(X)$ has a countable base
$E_{0} \supset E_{1}\supset E_{ 2} \cdots$~.
\end{definition}

Note that if an exterior space
$X$ is first countable at infinity, then
$\Co(X)$ is first countable at infinity.

\begin{theorem} \label{c2complete}Let $X$ be a locally path-connected exterior space and suppose that $X$ is
first countable at infinity. Then,
$\Co(X)$ is  $\Co$-complete.
\end{theorem}

\begin{proof} Recall that $L(\Co(X))=\bigcap_{E\in \varepsilon(X)} W_{0}(E)= \incl_{0} (\F(X))$.
Notice that if $C$ is a path-component of $E$, $C$ is open in $X$. We consider
$W_{0}(C)=p_{0}(C) \cup \incl_{0} (\eta_{0}^{-1}(q_{0}(C)))$.
Let $E_{0} \supset E_{1}\supset E_{ 2} \cdots$ be a countable base of the externology $\varepsilon(X)$.
To prove that there is a path from
$x\in p_{0}(C)$ to $b \in \eta_{0}^{-1}(q_{0}(C))$, we can take points $x_{n} \in E_{n}\subset E$ and
paths $\alpha_{n}$ from $x_{n}$ to $x_{n+1}$ to construct an exterior map $\alpha \colon [0, \infty) \to X$
($[0,\infty)$ with the cocompact externolgy) and an induced map $(p_{0}\alpha)' \colon [0,1]\cong [0,\infty)\cup \{\infty\} \to \Co(X)$ which is a path
from $x$ to $b$ in $W_{0}(C)$. This implies that $W_{0}(C)$ is path-connected.
Then
$\pi_{0}(E)\to \pi_{0}(W_{0}(E))$ verifies that
the path-component of $W_{0}(E)$ that contains $C$ also contains $W_0(C)$ and it  is surjective. Since $X$ is locally path-connected we have that
$\incl_{0}  (\eta_{0}^{-1}(q_{0}(C)))$ is open and closed in $\incl_{0}  (\F(X))$. This implies that
the path-components of $\incl_{0}  (\F(X))$ are singletons. Thus we obtain that $\pi_{0}(E)\to \pi_{0}(W_{0}(E))$ is
injective. Since for every $E$, $\pi_{0}(E)\to \pi_{0}(W_{0}(E))$  is a bijection, we have that
$\F(X)\to \F(\Co(X))$ is bijective. We also have the commutative diagram
$$\xymatrix{L(X) \ar[r] \ar[d] & L( \Co(X) ) \ar[d] \\ \F(X) \ar[r]  \ar[ur]^{\incl_{0} }& \F ( \Co(X) )}$$

Since $\ \F(X) \to L( \Co(X) )$, and $\F(X)\to \F(\Co(X))$ are bijective, it follows that
$L( \Co(X) ) \to \F ( \Co(X) )$ is a continuous bijection. By the definition of topology and externology
in $\Co(X)$  it is easy to check  that $\Co(X)$ satisfies the condition given in Theorem \ref{bijectioncero}.
Then,  it follows that $\Co(X)$ is  $\Co$-complete.

\end{proof}

\begin{definition} A net $x_{\delta}$ in an exterior space $(X,\varepsilon(X))$  is said to be
a \emph{$\varepsilon(X)$-net}  if for every $E\in \varepsilon(X)$ there is $\delta_0 $ such that
for every $\delta \geq \delta_0$, $x_{\delta}\in E$~.
A net $x_{\delta}$ is  said to be
a \emph{$\pi_0$-$\varepsilon(X)$-net}  if for every $E\in \varepsilon(X)$ there is a path-component $C$ of $E$ and  there is $\delta_0 $ such that
for every $\delta \geq \delta_0$, $x_{\delta}\in C$~.
\end{definition}

 \begin{theorem}\label{lpcomplete} Let $X$ be a locally path-connected  exterior space and for $x\in L(X)$ and
 $E \in \varepsilon (X)$ denote $C_{E}(x)$ the path-component of $x$ in $E$.
Then, $X$ is  $\Co$-complete if and only if $X$ satisfies the
following conditions:
\begin{itemize}
  \item[{\rm (i)}]  for every $x\in L(X)$ and $U\in (\t_{X})_{x}$, there is $E \in \varepsilon (X)$ such that $C_{E}(x)\subset U$,
  \item[{\rm(ii)}]  for every $x, y \in L(X)$, $x\not = y$, there is $E \in \varepsilon (X)$ such that $C_{E}(x)\cap C_{E}(y)= \emptyset$,
  \item[{\rm(iii)}] if  $x_{\delta}$ is a $\pi_0$-$\varepsilon(X)$-net, then there is  $x\in L(X)$ such that
$x_{\delta}\to x$~.
\end{itemize}
 \end{theorem}

  \begin{proof}
  If $X$ is  a $\Co$-complete exterior space, it is easy to check (i),(ii) and (iii).
  Conversely, if $X$ verifies (i), (ii) and (iii), to prove that $X$ is  $\Co$-complete, we have that condition (i) and the fact that $X$ is locally path-connected imply  the corresponding condition given in Theorem \ref{bijectioncero}. Then,    it suffices to check
 that the canonical continuous map $e_{0}\colon L(X) \to \F(X)$ is
  a bijection.   Suppose that $a \in \F(X)$ and $a=\{q_{0}^{-1}\eta_{0}^{E}(a)|E \in \varepsilon (X)\}$. Take $x_{E}\in q_{0}^{-1}(\eta_{0}^{E}(a))$, then $x_{E}$ is a $\pi_0$-$\varepsilon(X)$-net. By (iii), there is  $x\in L(X)$ such that $x_{E}\to x$. It is easy to check that $e_{0}(x)=a$. This implies that
$e_{0}\colon L(X) \to \F(X)$ is surjective. We can also see that (ii) implies that $e_0$ is injective.

\end{proof}

An interesting class of Hausdorff compact $\Co$-complete spaces are given in the following result:

\begin{theorem} \label{completeex}
Suppose that $X$  is a locally path-connected compact  Hausdorff space and  $D\subset X$ is a closed totally disconnected subspace.
Taking $\varepsilon(X)=\{U | D\subset U, U \in \t_X\}$, then one has $X$ is  $\Co$-complete.
\end{theorem}

\begin{proof} In order to apply Theorem \ref{lpcomplete}, we are going to check that  conditions (i), (ii) and (iii)
are satisfied:

(i) If $x\in D \subset U=E$, it is obvious that   $C_U(x)\subset U$.

(ii) Under these topological conditions, given an open $U$ such that $D\subset U$, there is
an open $V$ such that $D\subset V \subset \cl(V) \subset U$, where $\cl(V)=\overline{V}$. Now
if $x\in L(X)=D$ one has that $\bigcap_{U\in \varepsilon(X)} C_U(x)=\bigcap_{U\in \varepsilon(X)} \cl(C_U(x))\subset D$. Since the inverse limit of continua is a continuum, we have that $\bigcap_{U\in \varepsilon(X)} C_U(x)$ is a connected subset of $D$. Taking into account that $D$ is totally disconnected, one has that
$\bigcap_{U\in \varepsilon(X)} C_0^U(x)=\{x\}$. This implies condition (ii).

(iii) Suppose that $x_\delta$ is a $\pi_0$-$\varepsilon(X)$-net. From the definition of $\pi_0$-$\varepsilon(X)$-net, for each $U\in \varepsilon(X)$ there is a path-componente $C_U$ and $\delta_U$ such that $x_\delta \in C_U$
for $\delta \geq  \delta_U$. This implies that the family $\{\cl(C_U)\}$ satisfies the finite intersection property.
Since $X$ is compact, we have that $\bigcap_{U}\cl(C_U)=\bigcap_{U}C_U$ is non empty. It is easy to
check that if $x\in \bigcap_{U}C_U\subset D$, we have that  $x_\delta\to x.$

\end{proof}

  \begin{theorem}\label{condHausdorff} Let $X$ be a locally path-connected exterior space.
   \begin{itemize}
 \item[(i)]  If $x,x'\in L(\C_{0}(X))$, $x\not = x'$, there are open subsets $W,W'$ in $\Co(X)$ such that $x\in W$, $x'\in W'$ and $W\cap W' =\emptyset$. In particular, we have that $\F(X) \cong L(\C_{0}(X))$ is a Hausdorff space.
  \item[(ii)]  If $X$ is a Hausdorff space and  for every $x \in X \setminus L(X)$ there is a closed neighborhood $F$ at $x$ such that
  $X\setminus F$ is exterior, then  $\Co(X)$ is a Hausdorff space. In this case $L(X)$ is a closed subset of $X$.
\end{itemize}
 \end{theorem}

 \begin{proof}
 (i) If $x\not =x'$, $x=\incl_0(a),\, x'=\incl_0(a')$, then there is $E\in\varepsilon(X)$ such that
$\eta_0(a)\not = \eta_0(a')$. If $C=q_0^{-1}(\eta_0(a)), C'=q_0^{-1}(\eta_0(a))$, since  $X$ is locally path-connected
we have that $C, C'$ are open and $C\cap C'=\emptyset$. This implies that
$x\in W_0(C), x'\in W_0(C'), W_0(C)\cap W_0(C')=\emptyset$.

(ii) We can complete the cases analyzed in (i) as follows: If $x\not =x'$ and $\{x,x'\}\cap L(\Co(X))=\emptyset$,
then $x=p_0(\tilde x), x'=p_0(\tilde x') $. Now by the hypothesis of (ii) we can construct open subset
$U, U'$ of $X$ such that $\tilde x \in U,\, \tilde x' \in U',\, U\cap U'=\emptyset$ $(U\cup U')\cap L(X)= \emptyset$.
This implies that the  open subsets $ p_0(U),  p_0(U')$ separate $x,x'$. In the case $x\not \in L(\Co(X))$ and
$x' \in  L(\Co(X))$, then $x=p_0(\tilde x)$ and $\tilde x\not \in L(X)$. By hypothesis, there is a closed neighborhood $F$ at $\tilde x$ such that
  $X\setminus F$ is exterior. Then, $p_0(\mbox{int}F)$ and $W_0(X\setminus F)$ separate $x, x'$.
 \end{proof}

\begin{proposition} Let $X$ be an exterior space.
\begin{itemize}
\item[{\rm(i)}] If $x_{\delta}$ is a $\pi_0$-$\varepsilon(X)$-net, then there is   $x\in L(\Co(X))$ such that
$p_{0}(x_{\delta})\to x$.
\item[{\rm(ii)}] Suppose that $X$ is locally path-connected.  If $x_{\delta}$ is a $\pi_0$-$\varepsilon(X)$-net, then there is  a unique   $x\in L(\Co(X))$ such that
$p_{0}(x_{\delta})\to x$.
\end{itemize}
\end{proposition}

\begin{proof}(i) Suppose that $x_{\delta}$ is a $\pi_0$-$\varepsilon(X)$-net, then for each $E\in \varepsilon(X)$, there is
a path-component $C_E$ of $E$ and $\delta_E $ such that for every $\delta \geq \delta_E $, $x_{\delta}\in C_E$. If is easy to check that if $E'\subset E $, then $C_{E'} \subset C_{E}$. This implies that $a=(C_E)_{E\in \varepsilon(X)} \in \F(X)$. Take $x=\incl_0(a)\in \incl_0(\F(X))=L(\Co(X))$ and suppose that $W$ is a open neighbourhood at $x$ in $\Co(X)$. Then there is $E \in \varepsilon(X)$ and  an open $G$ of $\pi_0(E)$ such that $q_0^{-1}(G)\subset p_0^{-1}(W)$ and $a\in \eta_0^{-1}(G)\subset \incl_0^{-1}(W)$. This implies that $C_E\subset q_0^{-1}(G)\subset p_0^{-1}(W)$, therefore $p_0(x_{\delta}) \in W$ for
every $\delta \geq \delta_E$. Then $p_0(x_{\delta})\to x$.

(ii) Now we suppose that  that $X$ is locally path-connected and $p_0(x_{\delta})\to x,\, p_0(x_{\delta})\to x', \,
x, x' \in L(\Co(X)) $. If $x\not =x'$, by (i) of Theorem \ref{condHausdorff}, $x, x'$ can be separated by disjoint open subsets $W,W'$.
This fact  contradicts that $p_0(x_{\delta})\to x,\, p_0(x_{\delta})\to x'$.
\end{proof}

 \begin{remark} There are $\Co$-complete  locally path-connected exterior   spaces  $X$ having a  $\pi_0$-$\varepsilon(X)$-net  $x_{\delta}$ such that
 there is   $y\in X\setminus L(X)$ such that
$x_{\delta}\to y$. Nevertheless, there is a unique $x\in L(X)$ with $x_{\delta}\to x$.
 \end{remark}

 \begin{remark} It is remarkable that if $X$ is a $\Co$-complete exterior space, one has that $L(X)$ is a ``weak attractor'' of $X$ in
the sense that every $\pi_0$-$\varepsilon(X)$-net  $x_{\delta}$
has a limit point in $L(X)$. If $X$ is also locally
path-connected, then $x_{\delta}$ may have different limit points,
but there is a unique limit point in $L(X)$. If, in addition,  $X$
is also Hausdorff, then
  $L(X)$ is closed and every $\pi_0$-$\varepsilon(X)$-net  has a
unique limit point in $X$. Then
   $L(X)$ is an ``attractor'' in the sense that every $\pi_0$-$\varepsilon(X)$-net  $x_{\delta}$ has a unique limit point and this limit point is in $L(X)$. In this case, since $L(X)=L(\Co(X))\cong \F(X)$, one has that
   $L(X)$ is a Hausdorff,  totally disconected space (a prodiscrete space).
 \end{remark}

Next we analyze the compactness properties of completions:

 \begin{theorem}\label{condcompact} Let $X$ be an exterior space.
Suppose that  for every $E\in \varepsilon(X)$, there is $E' \in \varepsilon(X)$ such that
$E' \subset E$ and   the image of $\pi_{0}(E')\to \pi_{0}(E) $ is finite. Then,

   \begin{itemize}
  \item[(i)] $L(\C_{0}(X))$ is compact,
       \item[(ii)]   Denote $\eta_{0,E}(\F(X))$ the image of $\eta_{0,E}\colon \F(X)\to \pi_{0}(E)$. If for every $E\in \varepsilon(X)$,
       $X\setminus (\bigcup_{C \in \eta_{0,E}(\F(X))}C) $ is compact,  then $\Co(X)$ is compact.

 \end{itemize}
 \end{theorem}

\begin{proof}(i) The hypothesis condition implies that $\F(X)$ is a profinite space
(i.e., an inverse limit of finite discrete spaces). Since a profinite space is compact
and $L(\Co(X))=\incl_0(\F(X))$, we have that $L(\Co(X))$ is compact.

(ii) Take an open covering $\{W_i|i\in I\}$ of $\Co(X)$. For each $x\in L(\Co(X))$, there is $W_{i_x}$ such that
$x\in W_{i_x}$. By the definition of the topology $\mathcal{G}_0$, there is $E_{x} \in \varepsilon(X)$ and
$G_x$ open in $\pi_0(E_x)$ such that $x\in W_{0}(q_{0}^{-1}(G_x))\subset W_{i_x}$. Since $L(\Co(X))$ is compact,
there is a finite set $\{x_1, \cdots x_k\}$ such that $L(\Co(X))\subset  W_{0}(q_{0}^{-1}(G_{x_1})) \cup \cdots \cup W_{0}(q_{0}^{-1}(G_{x_k}))$. Take $E= E_{x_1}\cap \cdots \cap E_{x_k} \in \varepsilon(X)$ and $G'_{x_i}$  the inverse image
of $G_{x_i}$ via the map $\pi_0(E)\to \pi_0(E_{x_i})$. Then, one has $$L(\Co(X))\subset  W_{0}(q_{0}^{-1}(G'_{x_1})) \cup \cdots \cup W_{0}(q_{0}^{-1}(G'_{x_k}))\subset W_{x_1}\cup \cdots \cup W_{x_k} .$$
Since by hypothesis $X\setminus (q_{0}^{-1}(G'_{x_1})\cup \cdots \cup q_{0}^{-1}(G'_{x_k}))$ is a closed compact subset of $X$ we have that $p_0(X\setminus (q_{0}^{-1}(G'_{x_1})\cup \cdots \cup q_{0}^{-1}(G'_{x_k})))$ is compact.
Taking into account that
$$\Co(X)=(W_{0}(q_{0}^{-1}(G'_{x_1})) \cup \cdots \cup W_{0}(q_{0}^{-1}(G'_{x_k}))) \cup p_0(X\setminus (q_{0}^{-1}(G'_{x_1})\cup \cdots\cup q_{0}^{-1}(G'_{x_k})))$$
the finite family $\{W_{i_{x_1}}, \cdots, W_{i_{x_k}}\}$ together with a finite family covering of $p_0(X\setminus (q_{0}^{-1}(G'_{x_1})\cup \cdots \cup q_{0}^{-1}(G'_{x_k})))$
give a finite subcovering of the space $\Co(X)$.

\end{proof}

\begin{theorem}\label{main} Suppose that $X$ is a  locally path-connected, connected,  Hausdorff exterior space and $ \varepsilon(X)\subset \varepsilon^{\e}(X)$.
If for every
$x\in X\setminus L(X)$ there is a closed neighbourhood $F$ such that $X\setminus F$ is exterior,
then $\Co(X)$ is a Hausdorff  compact space and $L(X)$ is a closed subspace.
\end{theorem}
\begin{proof} Firstly we see, that under these conditions we also have that if $K$ is a closed compact subset of
$X$ contained in $X\setminus L(X)$, then $X\setminus K\in \varepsilon(X)$. Indeed, for each $x \in K$ there is
a closed compact neighborhood $F_x$ at $x$ such that $X\setminus F_x \in \varepsilon(X)$.
Since $K$ is compact, there is a finite $\{x_1,\cdots, x_k\}$ such that $K \subset  F_{x_1}\cup \cdots \cup F_{x_k}$.
This implies that $X\setminus K \supset X\setminus F_{x_1}\cap \cdots \cap X\setminus F_{x_k}$ and
therefore $X\setminus K\in \varepsilon(X)$.
Notice that we also have proved that if $K$ is a closed compact subset of
$X\setminus L(X)$, there is a closed compact  $K'=F_{x_1}\cup \cdots \cup F_{x_k}$ ($X\setminus K'\in \varepsilon(X)$) such that
$K \subset $int$(K')$.

Now, given $E_1\in \varepsilon(X)$, we are going to prove that there is
$E_2 \in \varepsilon(X)$ such that $\cl(E_2) \subset $int$(E_1)$. Notice that the condition
$ \varepsilon(X)\subset \varepsilon^{\e}(X)$ implies that
the frontier  $\Fr(E_1)$ is a closed compact contained in $X\setminus L(X)$.
By the argument above, there is  $K'$ such that $\Fr(E_1)\subset$int$(K')$ and
taking  $E_2=E_1\setminus K' \in \varepsilon(X)$,  we have
that $\cl(E_2) \subset E_1$.

Suppose that $C_2$ is a path-component of $E_2$ and $C_1$ the
path-component of $E_1$, such that $C_2 \subset C_1$. If $C_1 \cap \Fr(E_2)=\emptyset$, consider
$x_1\in C_1$ and  a path $\alpha \colon [0,1]\to C_1\subset E_1$ such that
$\alpha(0)=x_1$, $\alpha(1)=x_2 \in C_2$. Note that $\alpha^{-1}(E_2)=\alpha^{-1}(\cl(E_2))$ is a non-empty open and closed subset of $[0,1]$. This implies that $\alpha^{-1}(E_2)=[0,1]$; that is, $\alpha([0,1]) \subset  E_2$. Thus
one has that $x_1\in C_2$ and
$C_1=C_2$. Since $C_1$ is a closed subset of $E_1$ and $\cl(E_2)\subset E_1$, one has that   $C_1$ is an open and closed subset of $X$, contradicting the fact that $X$ is connected.  Then we have
that $C_1 \cap \Fr(E_2)\not=\emptyset$. Since $\Fr(E_2)\subset \bigcup_{C \in \pi_0(E_1)} C$ and
$\Fr (E_2)$ is compact, we obtain that there is a finite number of components having non empty intersection with
$\Fr (E_2)$. This implies that the image of $\pi_0(E_2)\to \pi_0(E_1)$ is finite.

We note that Im$(\F(X)\to \pi_0(E_1))\subset$Im$(\pi_0(E_2)\to \pi_0(E_1))$ is a finite set.

One the other hand, suppose that $E\in \varepsilon(X)$ and $C \in
\pi_0(E)$. It is easy to check that $C\in$Im$(\F(X)\to \pi_0(E))$
if and only if for every $E'\in \varepsilon(X)$ such that
$E'\subset E$, $C\cap E' \not = \emptyset$. Using this
characterization, the fact that $\varepsilon(X)$ is closed by
finite intersections and the  images above are finite, it follows
that there exists  $E_3\in \varepsilon(X)$ such that $E_3\subset
E_1$, Im$(\F(X)\to \pi_0(E_1))=$Im$(\pi_0(E_3)\to \pi_0(E_1))$ and
$E_3 \subset \bigcup_{C\in \eta_{0,E_1}(\F(X))}C$. Since
$X\setminus E_3$ is compact and $\bigcup_{C\in
\eta_{0,E_1}(\F(X))}C$ is open, we have that $X\setminus
(\bigcup_{C\in \eta_{0,E_1}(\F(X))}C)$ is compact.

Now applying  Theorems \ref{condHausdorff} and  \ref{condcompact} one has that $\Co(X)$ is a Hausdorff compact space and $L(\Co(X))$ is a closed subspace.

\end{proof}

\begin{corollary} Suppose that $X$ is a locally compact, locally path-connected, connected,  Hausdorff space and $ \varepsilon(X)=\varepsilon^{\e}(X)$.  Then $L(X)=\emptyset $, $\F(X)=\mathcal{F}(X)$ is the space of Freudenthal ends of $X$ and $\Co(X)$ is a Hausdorff  compact space. \end{corollary}

\begin{remark}
We point out that, with the hypothesis of the corollary above, the
underlying topological space of $\Co(X)$ is exactly the
Freudenthal compactification of $X$, see \cite{Freudenthal}.
 \end{remark}

\section{The category of $\r$-exterior flows}

In this section we are going to consider the exterior space $\R^{\r}=(\R, \r)$.
We recall  (see subsection \ref{exteriorspaces})  that  $\r$  is  the following   externology:
$${\r}=\{ U| U\hspace{3pt}\mbox{is open and there is}\hspace{3pt}
n\in \N\hspace{3pt}\mbox{such that} \hspace{3pt}(n, +\infty)
\subset U\}.$$

The exterior space $\R^{\r}$
plays an important role in the definition of $\r$-\emph{exterior
flow} below.
Such notion mixes the structures of dynamical system
and exterior space (see \cite{EHR10},\cite{GHR12}):

\begin{definition}  Let $M$ be an exterior space, $M_{\t}$ denote
the subjacent topological space and $M_{\d}$ denote the set $M$
provided with the discrete topology. An  $\r$-\emph{exterior flow}
is a continuous flow $\varphi \colon \R {\times} M_{\t} \to
M_{\t}$ such that $\varphi \colon \R^{\r}\bar\times M_{\d} \to  M$
is exterior and for any $t\in \R$, $F_t\colon M \bar \times I \to M$,
$F_t(x,s)=\varphi (ts,x)$, $s\in I,~x \in M$, is
also exterior.

An  $\r$-\emph{exterior flow morphism} of $\r$-exterior flows
is a flow morphism $f\colon M \to N$   such that $f$ is exterior.
We will denote by $\bf{E^{\r} F}$ the category of $\r$-exterior
flows and $\r$-exterior flow morphisms.
\end{definition}

Given an $\r$-exterior flow $(M,\varphi) \in  \bf{E^{\r} F}$, one
also has a flow $(M_{\t},\varphi) \in \bf{F}.$ This gives a
forgetful functor $$(\cdot)_{\t} \colon \bf{E^{\r} F} \to
\bf{F}.$$

Now given a continuous flow  $X=(X,\varphi)$, an open $N \in \t_X$ is said to
be ${\r}$-\emph{exterior} if for any $x \in X$ there is $T^x \in
\r$ such that $\varphi (T^x \times \{x\}) \subset N.$  It is easy
to check that the family of $\r$-exterior subsets of $X$  is an
externology, denoted by $\varepsilon^{\r}(X),$ which gives an
exterior space $X^{\r}=(X, \varepsilon^{\r}(X))$  such that  $\varphi \colon \R^{\r}\bar \times
X_{\d} \to X^{\r} $ is  exterior
and $F_t \colon X^{\r} \bar \times I \to  X^{\r}$, $F_{t}(x,s)=\varphi( ts, x)$,  is also exterior for every
$t\in\R$.
Therefore $(X^{\r},\varphi)$ is an $\r$-exterior flow
which   is said to be the $\r$-\emph{exterior flow associated
to} $X.$ When there is no possibility of confusion, $(X^{\r},
\varphi )$ will be briefly denoted by $X^{\r}$.  Then we have a
functor
$$(\cdot)^{{\r}} \colon \bf { F}\to \bf{E^{\r} F}.$$

The category of flows can be considered as a full subcategory of the
category of exterior flows:

\begin{proposition} The functor $(\cdot)^{{\r}} \colon \bf{ F}\to \bf{E^{\r} F}$
is left adjoint to the functor $(\cdot)_{\t} \colon \bf{E^{\r}
F}\to \bf{F}.$ Moreover $(\cdot)_{\t} \, (\cdot)^{{\r}}={\rm id}$
and $\bf{F}$ can be considered as a full subcategory of
$\bf{E^{\r} F}$ via $(\cdot)^{\r}.$
\end{proposition}

\subsection{End Spaces and Limit Spaces of an exterior flow}

In section \ref{tres} we have defined the end and limit spaces of
an exterior space. In particular, since any $\r$-exterior flow $X$
is an exterior space, we can consider the end space
 $\F(X)$
and the limit space $L(X)$.
Notice that one has the following properties:

\begin{proposition} \label{invariantes} Suppose that $X=(X,\varphi)$ is an
$\r$-exterior flow. Then
\begin{itemize}
\item[(i)] The space  $L(X)$
is invariant.
\item[(ii)] There is a trivial  flow structure induced  on $\F(X).$
\end{itemize}
\end{proposition}

\begin{proof}(i) We have that $L(X)=\cap_{E\in \varepsilon(X)}E.$ Note that
for any $s\in \R,$ $\varphi_s(E)\in \varepsilon(X)$ if and only if
$E\in \varepsilon(X).$ Then $\varphi_s(L(X))=\varphi_s( \cap_{E\in
\varepsilon(X)}E)= \cap_{E\in
\varepsilon(X)}\varphi_s(E)=\cap_{E\in \varepsilon(X)}E=L(X).$

(ii) For any $s\in \R$, consider the exterior homotopy $F_{s}\colon
X\bar \times I \to X$, $F_{s}(x,t)= \varphi(t s,x)$, from $\id_X$ to
$\varphi_s$. By Lemma \ref{lemahomotopia}, it follows that $\id =
\F(\varphi_s)$.

\end{proof}

As a consequence of this result, one has functors
 $L, \F
 \colon {\bf E^{\r}F} \to {\bf F}.$

\begin{proposition}\label{homotopy} The functors
$L, \F
\colon {\bf E^{\r}F} \to {\bf F}$ induce functors
$$L
\colon \pi {\bf E^{\r}F} \to \pi{\bf F},\hspace{10pt}\F
\colon\pi{\bf E^{\r}F} \to {\bf F},$$ where the homotopy categories
are constructed in a canonical way.
\end{proposition}

\subsection{The end point of a trajectory and the induced decompositions of
an exterior flow}

For an $\r$-exterior flow $X$, one has that  each trajectory has an
end point given as follows: Given $p \in X$ and $E \in
\varepsilon(X),$ there is $T^p\in \r$ such that $T^p \cdot p
\subset E.$ We can suppose that $T^p$ is path-connected and
therefore so is $T^p \cdot p;$ this way there is a unique
path-component $\omega_{\r}^0(p,E)$
of $E$ such that $T^p \cdot p \subset
\omega_{\r}^0(p, E) \subset E$.
This gives maps $\omega_{\r}^0
(\cdot , E) \colon X\to \pi_0(E)$ and $\omega_{\r}^0 \colon X \to
\F(X)$
such that the following diagram
commute:

$$\xymatrix{ L(X) \ar[r]^{e_0} \ar[d] & \F(X)\\
X\ar[ur]^{\omega_{\r}^0} & }$$

These maps permit to divide a flow in simpler subflows.

\begin{definition} Let $X$ be an  $\r$-exterior flow. We will
consider $X_{(\r,a)}^0=(\omega_{\r}^0)^{-1}(a),$ $a\in \F(X)$.
The invariant space $X_{(\r,a)}^0$
will be called the
$\omega_{\r}^0$-\emph{basin at} $a\in \F(X)$.

The map $\omega_{\r}^0$
 induces the following
partition of $X$ in simpler flows
$$X=\bigsqcup_{a\in \F(X) }X_{(\r, a)}^0$$
that will be called
respectively, the $\omega_{\r}^0$-\emph{decomposition}
of the $\r$-exterior flow $X.$
\end{definition}

It is important to note that the map $\omega_{\r}^0$
needs not
be continuous.

\begin{definition} Let $X$ be an  $\r$-exterior flow. An end point $a\in \F(X)$
is said to be ${ \omega_{\r}^0 }$-\emph{representable}
if there is $p\in X$ such that $ \omega_{\r}^0 (p)=a.$
Denote by $ \omega_{\r}^0 (X)$ the space of
$\omega_{\r}^0$-representable end points.
\end{definition}

Since the
$\omega_{\r}^0$-decomposition of $X$ is compatible with the
$e_0$-de\-com\-position of the limit subspace, we have a
canonical map $e_0L(X) \to \omega_{\r}^0 (X)$.

\section{The completion of an $\r$-exterior flow}\label{completionexterior}

In this section, we use the completion functor $\Co \colon {\bf E} \to {\bf E}$ given in subsection
\ref{exteriorcompletion} to construct the completion functor  $\Co \colon {\bf E^{\r}F} \to {\bf E^{\r}F}$.

Given an $\r$-exterior flow $X=(X, \varphi)$, we can take the following diagram in the category of topological spaces ${\bf Top}$,
where the front and back faces are push-outs in {\bf Top}
$$\xymatrix{& L(X) \ar[rr]^{e_0} \ar[dd] & & \F(X) \ar[dd]^{\inc_0}\\
\R \times L(X) \ar[ur]^{\varphi|_{\R \times L(X)}} \ar[rr]^(0.6){\id \times e_0} \ar[dd] & &\R \times  \F(X)\ar[ur]^{\pr} \ar[dd] &\\
& X \ar[rr]^(0.4){p_0} && X\cup_{L(X)}\F(X)\\
\R \times X \ar[ur]^{\varphi} \ar[rr]^{\id \times p_0} &&\R \times ( X\cup_{L(X)}\F(X)) \ar@{.>}[ur]^{{\check \varphi}_{0} }& }
$$
Therefore, considering such push-out topologies, there is  an induced continuous map

$${\check \varphi}_{0}\colon\R \times ( X\cup_{L(X)}\F(X))\to X\cup_{L(X)}\F(X)$$

In order to prove that ${\check \varphi}_{0}$ is continuous with the topology $\mathcal{G}_0$
we introduce some notation and study some properties:

Given $S \subset X$ and $t\in \R$, denote by
$$S^t=\{q\in S| st \cdot q \in S, ~\, 0\leq s\leq 1\}$$
We note that if $S_1\subset S_2$, then $S_1^t \subset S_2^t$~.
\begin{lemma} Suppose that $U$ is an open in a flow $X$ and $t\in \R$. Then,
\begin{itemize}
  \item[(i)] $U^t$ is an open subset of $X$,
  \item[(ii)] if $A$ is a path-component of $U^t$, $B$ is a path-component of $U$ and
  $A    \subset B$, then $A\subset B^t$.
\end{itemize}
\end{lemma}

\begin{proof} (i) Suppose that $q\in U^t$; then for each $0\leq s \leq 1$ there are $W(s)\in(\t_{[0,1]})_s$ and
$V(s) \in (\t_X)_q$ such that $W(s)t\cdot V(s) \subset U$. Since $[0,1]$ is compact,
we can find $V\in (\t_X)_q$ such that for all  $0\leq s'\leq 1$, $s't \cdot V \subset U$. This implies that
$V\subset U^t$. Therefore $U^t$ is open.
(ii) Take $q\in A$, then $q \in  ([0,1] t)\cdot q \subset U$ and $ ([0,1] t)\cdot q$ is path-connected, then
we have that $ ([0,1] t)\cdot q \subset B$. This implies that $A\subset B^t$~.
\end{proof}

\begin{lemma}\label{bete}  Suppose that $f\colon X \to Y$ is a morphism of $\R$-sets.
If $B\subset Y$, then $f^{-1}(B^t)=(f^{-1}(B))^t.$
\end{lemma}
\begin{proof}
Just take into account that $f(st\cdot x)=st\cdot f(x).$
\end{proof}

\begin{proposition}  Let $X=(X, \varphi)$  be an $\r$-exterior flow. Then,
${\check \varphi}_{0} $ induces on $\Co(X)$ an $\r$-exterior flow structure; that is:
\begin{itemize}
\item[\rm (i)] the map ${\check \varphi}_{0} \colon \R \times \Co(X) \to \Co(X)$ is continuous (with the topology $\mathcal{G}_0$)
\item[\rm (ii)] the map ${\check \varphi}_{0}  \colon \R^{\r} \bar \times \Co(X)_{\d} \to \Co(X)$ is exterior, and
\item[\rm (iii)]    for any $t\in \R$, $F_t\colon \Co (X)\bar \times I \to \Co (X)$,
$F_t(x,s)={\check \varphi}_{0}  (ts,x)$, $s\in I,~x \in \Co (X)$, is
also exterior.
\end{itemize}
\end{proposition}

\begin{proof} (i) Firstly, recall  that if $f\colon X\to Y$ is an exterior map,
by Proposition \ref{functor},
$\Co(f)$ is a continuous map (in fact is an exterior map).
In particular, given  $X=(X, \varphi)$  an $\r$-exterior flow,  since $\varphi_t$ is an exterior map and
$ ({\check \varphi}_{0})_t=\Co(\varphi_t)$,  one has that $ ({\check \varphi}_{0})_t$ is a continuous map.

We observe  that if  $Y$ is  a topological space and there is  a set action
$\phi \colon \R \times Y \to Y$ such that each $\phi_t$ is continuous and $\phi$ is continuous on
$\{0\} \times Y$, then $\phi$  is continuous. Therefore, in our  case,
it suffices to check that the action ${\check \varphi}_{0}\colon \R \times \Co (X) \to \Co (X)$ is continuous on
$\{0\} \times \Co (X)$; that is, if $0\cdot y = y \in W$ with $W$ an open subset of  $\Co (X) $, then there exist
$\delta>0$ and an open $W'$ containing $y$ such that $[-\delta, \delta]\cdot W'\subset W$.

Before proving the continuity we analyze some previous properties:

Suppose that $W\subset \Co (X)$ is open and take $t\in \R$.
By Lemma \ref{bete}, $ (p_0^{-1}(W))^{t}=p_0^{-1}(W^{t})$.
We also have that $\incl_0^{-1}(W^{t})=(\incl_0^{-1}(W))^{t}=\incl_0^{-1}(W)$ since the action on $\F(X)$ is trivial.

We know that $\incl_0^{-1}(W)$ is open in $\F(X)$ and for each
$a \in \incl_0^{-1}(W^{t})= \incl_0^{-1}(W)$  there is $E\in \varepsilon(X)$ and $G$ open in $\pi_0(E)$ such that
$a\in \eta_0^{-1}(G)\subset \incl_0^{-1}(W)$,
$q_0^{-1}(G) \subset p_0^{-1}(W)$.  Taking $E^{t}$ and $(\eta_E^{E^t})^{-1}G$, one has that
$a\in \eta_0^{-1}((\eta_E^{E^t})^{-1}G)\subset \incl_0^{-1}(W^t)$, $q_0^{-1}((\eta_E^{E^t})^{-1}G) \subset (p_0^{-1}(W))^{t}$. This implies that $W^t$ is also open. Now, taking the definition of $W^t$, one has that for any $t>0$
$[-t,t]\cdot (W^t\cap W^{-t}) \subset W$.

Now to prove that ${\check \varphi}_{0}$ is continuous at $(0,y)$ we distinguish two cases:

a) Suppose that $y \in W\subset \Co (X)$ with $y=p_0(x)$ and
$x\in p_0^{-1}(W)$. Then, there are $\varepsilon >0$ and $U \in (\t_X)_x$ such that
$[-\varepsilon, \varepsilon]\cdot U \subset p_0^{-1}(W)$.
This implies that $y\in W^{\varepsilon}\cap W^{-\epsilon} $ and therefore
$[-\varepsilon,\varepsilon]\cdot (W^{\varepsilon}\cap W^{-\varepsilon}) \subset W$.

b) Otherwise,  $y\not \in p_0(X)$, so $y\in \incl_0(\F(X))$ ($y$ is not $e_0$-representable). In this case we have that  for any $t>0$, $y \in W^t$  and $[-t,t]\cdot (W^t\cap W^{-t}) \subset W$.

The proof of (ii) and (iii) is a routine checking.

\end{proof}

\begin{definition} Given  an $\r$-exterior flow $X=(X, \varphi)$, the $\r$-exterior flow $\Co(X)=(\Co(X), {\check \varphi}_{0}  )$ is called the \emph{$\Co$-completion} of $X$.
\end{definition}

Taking into account that for all $E\in \varepsilon(X)$,
$p_0(E) \subset W_0(E)$ the following result holds.

\begin{proposition} The construction $\Co \colon {\bf E^{\r}F}\to {\bf E^{\r}F}$ is a functor and the
canonical map $p_0: X\to \Co(X) $ induces a natural transformation.

\end{proposition}

Notice that $\R$ with the usual order is a directed set. Then, if
$X$ is a flow,  $\varphi^x \colon \R \to X$ is a net of the space $X$, for each $x\in X$.

\begin{theorem}\label{redes} Let $X\in {\bf E^{\r}F} $ be an $\r$-exterior flow. Then,
\begin{itemize}
  \item[\rm (i)]  For any $x  \in X$, $\varphi^x$ is a $\pi_0$-$\varepsilon(X)$-net.
  \item[\rm (ii)]   For any $x  \in X$, $ p_{0}\varphi^x\to \omega_r^0(x)$ in $\Co(X)$ ($\omega_r^0(x)$ is identified to  $\incl_{0}(\omega_r^0(x))$).
  \item[\rm (iii)]  If $X$ is locally path-connected  and  $p_{0}\varphi^x\to a$, $p_{0}\varphi^x\to b$, $a,b \in \F(X)\subset \Co(X)$, then $a=b$.
\end{itemize}
\end{theorem}

\begin{proof}(i)  Given $E\in \varepsilon(X)$, since $\varphi^x$ is exterior, there is $n_E\in \R$ such that
$\varphi^x([n_E, \infty))\subset E$. Since $[n_E, \infty)$ is path-connected, there is a path-component $C_E$ of $E$
such that $\varphi^x([n_E, \infty))\subset C_E$. This implies that $\varphi^x$ is a $\pi_0$-$\varepsilon(X)$-net.

(ii) Notice that $\omega_r^0(x)=(C_E)_{E\in\varepsilon(X)}$.    If $\incl_0 (\omega_r^0(x)) \in W$ and
$W$ is open, there is $E\in \varepsilon(X)$ and $G$ open in $\pi_0(E)$ such that
$q_0^{-1}(G)\subset p_0^{-1}(W)$ and $\omega_r^0(x) \in \eta_{0,E}^{-1}(G)\subset \incl_0^{-1}(W)$.
Then $\varphi^x([n_E, \infty))\subset C_E\subset q_0^{-1}(G) \subset p_0^{-1}(W)$ and therefore $p_0\varphi^x([n_E, \infty))\subset W$.

(iii) Since $X$ is locally path-connected, by Theorem \ref{condHausdorff} we have that $\F(X)$ is a Hausdorff space. This implies that
$p_0\varphi^x$ has a unique limit within $\incl_0(\F(X))=L(\Co(X))$.

\end{proof}

If we apply the theorem above to a $\Co$-complete  $\r$-exterior flow, we obtain the next result.

\begin{theorem} \label{limit} Let $X\in {\bf E^{\r}F} $ be a $\Co$-complete  $\r$-exterior flow. Then,
\begin{itemize}
  \item[\rm (i)]   For any $x  \in X$, $\varphi^x\to e_{0}^{-1}\omega_{\r}^0(x)$ and therefore $e_{0}^{-1}\omega_{\r}^0(x)
  \in \omega^{\r}(x)$.
    \item[\rm  (ii)]  If $X$ is Hausdorff, then $\omega^{\r}(x)=\{e_{0}^{-1}\omega_{\r}^0(x)\}$ and $C(X)=P(X)=\overline{\Omega^{\r}(X)}.$
\end{itemize}
\end{theorem}

\begin{proof}(i) By Theorem \ref{redes}, $p_0\varphi^x \to \incl_{0}(\omega_r^0(x))$, so
$\varphi^x \to p_0^{-1} \incl_{0}(\omega_r^0(x))=e_{0}^{-1}(\omega_r^0(x))$.

(ii) Using the convergence properties of nets and subnets and the fact that $X$ is Hausdorff, one has
$\omega^{\r}(x)=\{e_{0}^{-1}\omega_{\r}^0(x)\}$.
Since $\omega^{\r}(x)$ is invariant, it follows that $\Omega^{\r}(X) \subset C(X)$. Obviously,
one has that $ C(X)\subset P(X) \subset  \Omega^{\r}(X) $.
Therefore, it follows that $ C(X)=P(X)=\Omega^{\r}(X) $. Taking into account that  set of
critical points in a Hausdorff flow is closed, it is not difficult to check that $C(X)=P(X)=\overline{\Omega^{\r}(X)}.$

\end{proof}

Taking into account the definition of global weak attractor and global attractor, see Definition \ref{attractoruno}, the following is obtained:

\begin{corollary}\label{attractor} Let $X\in {\bf E^{\r}F} $ be a $\Co$-complete  $\r$-exterior flow. Then,
\begin{itemize}
  \item[\rm (i)]   $L(X)$  is a  global weak attractor of $X$ and $L(X)=C(X)$.
    \item[\rm (ii)]  If $X$ is Hausdorff,  then $L(X)$  is the unique  minimal global attractor of $X$ and $L(X)=\overline{\Omega^{\r}(X)}$.

\end{itemize}
\end{corollary}

\begin{proof}(i) By (i) of Theorem \ref{limit}, we have that for every $x\in X$, $e_{0}^{-1}\omega_{\r}^0(x)\in \omega^{\r}(x)$,
so $\omega^{\r}(x)\cap L(X)\not= \emptyset$. Thus, $L(X)$ is a global weak attractor. It is clear that
$C(X) \subset L(X)$. Since $X$ is $\Co$-complete, one has that $L(X)\cong \F(X)$ as flows. Taking into account
that the action on $\F(X)$ is trivial, it follows that every point of $L(X)$ is critical. We conclude $L(X)=C(X)$.

(ii) Assume that $M\subset L(X)$ is a global weak attractor. If $x\in L(X)\setminus M$,
since $X$ is $T_1$ and $x$ is critical, we have that
$$\omega^{\r}(x)= \bigcap_{t \geq 0}\overline{ [t, +\infty) \cdot x}=\overline{\{x\}}=\{x\}.$$
 This contradicts the minimality of $M$. On the other hand,
we have that $L(X)=C(X)\subset P(X) \subset \Omega^{\r}(X)$. But by Theorem \ref{limit}, $\omega^{\r}(x)=\{\omega_{\r}^0(x)\}\subset L(X)$, implying that $L(X)=C(X)=\Omega^{\r}(X)$. Finally, since in a Hausdorff flow the set of critical points $C(X)$ is closed,
one has that $L(X)=\overline{\Omega^{\r}(X)}$.

\end{proof}

\section{Ends, Limits  and completions  of a flow via exterior flows}\label{completion}

Recall that we have considered the functor
$$(\cdot)^{{\r}} \colon \bf { F}\to \bf{E^{\r} F},$$
the forgetful functor $$(\cdot)_{\t} \colon \bf{E^{\r} F} \to
\bf{F},$$
the functors $$L, \F
\colon {\bf E^{\r}F}
\to {\bf F},$$
and
$${\bar  L}, \Co
\colon {\bf E^{\r}F}
\to {\bf E^{\r}F}.$$
Therefore we can consider the composites:
$$L^{{\r}}:=L(\cdot)^{{\r}},\hspace{5pt}\F^{{\r}}:=\F(\cdot)^{{\r}},
\hspace{5pt}{\bar L}^{{\r}}:=(\cdot)_{\t} {\bar L}(\cdot)^{{\r}},\hspace{5pt}\Co^{{\r}}:=(\cdot)_{\t} \Co(\cdot)^{{\r}}
$$
to obtain functors $L^{{\r}},$
$\F^{{\r}},{\bar L}^{\r}, \Co^{{\r}}
\colon {\bf F} \to {\bf F}.$

In this way, given a flow $X$, we have the {\it end (trivial) flow}
$\F^{\r}(X)=\F(X^{\r})$,
the {\it limit flow}
$L^{\r}(X)=L(X^{\r})$, the {\it bar-limit flow} ${\bar L}^{\r}(X)={\bar L}(X^{\r})$ and the {\it completion flow}
$\Co^{\r}(X)=(\Co(X^{\r}))_{\t}$.

\bigskip It is interesting to consider the following equivalence of
categories: Given any flow $\varphi \colon \R \times X \to X,$ one
can consider the \emph{reversed flow} $\varphi' \colon \R \times X
\to X$ defined by $\varphi'( r, x)=\varphi (-r,x),$ for every
$(r,x) \in \R \times X.$ The correspondence $(X,\varphi)\to
(X,\varphi')$ gives rise to a functor
$$(\cdot)' \colon {\bf F} \to {\bf F}$$ which is an equivalence of
categories and verifies $(\cdot)' (\cdot)'= \id.$ Using the
composites
$$L^{{\l}}:=(\cdot)'L^{{\r}}(\cdot)',\hspace{5pt}
\F^{{\l}}:=(\cdot)' \F^{{\r}}(\cdot)',\hspace{5pt} {\bar L}^{\l}:=(\cdot)'{\bar L}^{\r}(\cdot)', \hspace{5pt} {\Co}^{\l}:=(\cdot)'{\Co}^{\r}(\cdot)', $$
we obtain the new functors
$L^{{\l}}, \F^{{\l}},{\bar L}^{\l}, \Co^{{\l}} \colon {\bf F}\to {\bf F}.$

\begin{definition} A flow $X$ is said to be \emph{$\Co^{\r}$-complete} (\emph{$\Co^{\l}$-complete}) if the canonical
map $p_0\colon  X \to \Co^{\r}(X)$ is an homeomorphism. A flow is said
\emph{$(\Co^{\r})^2$-complete} \emph{($(\Co^{\l})^2$-complete}) if $\Co^{\r}(X)$ ($\Co^{\l}(X)$) is $\Co^{\r}$-complete
($\Co^{\l}$-complete).
\end{definition}

We remark  that in this definition we denote  $(p_0)_{\t}$ by $p_0$. In fact, for a flow $X$, one has that
$p_0\colon X^{\r}\to \Co(X^{\r})$ is an isomorphism in ${\bf E^{\r}F}$ if and only if $p_0=(p_0)_{\t} \colon X \to \Co^{\r}(X)$
is a flow homeomorphism   in ${\bf F}$.

For the following results we will use some previous properties
whose proofs are given in Theorem 6.3, Lemma 6.11 and Corollary
6.14 of \cite{GHR12}, respectively.

\begin{proposition}\label{the} Let $X$ be a flow and suppose that $X$ is a
$T_1$-space. Then $$P(X)=L^{\r}(X).$$
\end{proposition}

\begin{proposition}\label{lem} Let $X$ be a flow and suppose that $X$ is a locally compact regular space.
If  $x \not \in\overline{ \Omega^{\r}(X)},$ then there exists $V_x
\in {({\t}_X)}_x$ such that  $X\setminus \overline{V}_x$ is
$\r$-exterior.
\end{proposition}

\begin{proposition} \label{cor} Let $X$ be a flow. If $X$ is a locally compact
$T_3$  space, then $L^{\r}(X) = P(X)$, $ \bar L^{\r}(X)=
\overline{\Omega^{\r}(X)}$ and $$L^{\r}(X)=P(X)\subset
P^{\r}(X)\subset \Omega^{\r}(X) \subset
\overline{\Omega^{\r}(X)}=\bar L^{\r}(X).$$
\end{proposition}

Then the following result holds:

\begin{theorem}\label{propcomplete} Let $X$ be a $\Co^{\r}$-complete  flow. Then,
\begin{itemize}
  \item[\rm (i)]   $L^{\r}(X)$  is a  global weak attractor  of $X$ and $L^{\r}(X)=C(X)$.
    \item[\rm (ii)]  If $X$ is $T_{1}$, then $L^{\r}(X)$  is a minimal global weak attractor  of $X$ and $L^{\r}(X)=P(X)$.

    \item[\rm (iii)]  If $X$ is $T_{2}$,  then $L^{\r}(X)$  is the unique  minimal global attractor  of $X$ and $L^{\r}(X)=\overline{\Omega^{\r}(X)}$
 \item[\rm (iv)]      If $X$ is  locally compact and  $T_{3}$, then  $L^{\r}(X)=  \overline{\Omega^{\r}(\Co(X))}= {\bar L^{\r}}(X)$.

\end{itemize}
\end{theorem}

\begin{proof}(i) It follows from (i) of Corollary \ref{attractor}.

(ii) Suppose that $M\subset L^{\r}(X)$ is a global weak attractor. Take $x\in L^{\r}(X)$;
since $X$ is $T_1$ and $x$ is critical, we have that $\omega^{\r}(x)=\{x\}$ and therefore   $x\in M$.
By Proposition \ref{the} one has that $L^{\r}(X)=P(X)$.

(iii) follows from (ii) of Corollary \ref{attractor}
and
(iv)  is a consequence of Proposition \ref{cor}.
\end{proof}

\begin{theorem} \label{cHausdorff} Let $X$ be a locally path-connected, locally compact $T_{2}$ flow.   Then,  $L^{\r}(X)= \overline{\Omega^{\r}(X)}$ if and only if $\Co^{\r}(X)$ is a $T_{2}$  flow.
\end{theorem}

\begin{proof}
Ir order to apply Theorem \ref{condHausdorff}, take $x\in X\setminus L^{\r}(X)$. Since $L^{\r}(X)= \overline{\Omega^{\r}(X)}$ by Theorem \ref{propcomplete}, we have that $x\not \in \overline{ {\Omega^{\r}(X)}}$. By Proposition \ref{lem}, there is $V_x\in (\t_{X})_x$ such that $X\setminus \overline{V_x}$ is $\r$-exterior. It follows that $\Co^{\r}(X)$ is a $T_{2}$  flow.

Conversely, suppose that $\Co^{\r}(X)=\Co(X^{\r})$ is a $T_{2}$ flow. By Theorem \ref{propcomplete},
$L(\Co(X^{\r}))=P(\Co(X^{\r})$.  We also have that $C(\Co(X^{\r}))\subset P(\Co(X^{\r})$
and in this case, since $P(\Co(X^{\r})=L(\Co(X^{\r}))\cong \F(X)$ has a trivial action,
we have that $P(\Co(X^{\r}) \subset C(\Co(X))$. Therefore $L(\Co(X^{\r}))=P(\Co(X^{\r})=C(\Co(X^{\r}))$.

We also have that
$C(\Co^{\r}(X))$ is a closed space because  $\Co^{\r}(X)$ is  $T_{2}$ .
This implies that $$L^{\r}(X)=p_{0}^{-1}(L(\Co^{\r}(X)))=p_{0}^{-1}(C(\Co^{\r}(X)))$$ is also closed. Now suppose
$\Omega^{\r}(X)\setminus L^{\r}(X)\not = \emptyset$, then there are
$\varphi^x$ and $y\in \Omega^{\r}(X)\setminus L(X^{\r})$ such that
$\varphi^x(t_{\delta})\to y$, $t_{\delta}\to +\infty$. This implies that
$p_{0}(\varphi^x(t_{\delta}))\to p_{0}(y)$ and $p_{0}(\varphi^x(t_{\delta}))\to \omega_{\r}^{0}(x)$
and $ \omega_{\r}^{0}(x)\not = p_{0}(y)$. Then $\Co^{\r}(X)$ is not  $T_{2}$, which is
a contradiction. Therefore, $L^{\r}(X)=\Omega^{\r}(X)$. Finaly, taking into account that
$L^{\r}(X)$ is closed we have that $L^{\r}(X)=\overline{\Omega^{\r}(X)}$.
\end{proof}

\begin{definition} A flow $X$ is said to be \emph{Lagrange stable} if for every $x\in X$ the semi-trajectory
$[0,\infty)\cdot x$ is contained in  a compact subspace.
\end{definition}

\begin{corollary} Let $X$ be a locally path-connected, locally compact $T_{2}$ flow and suppose that
$X$ is Lagrange stable.   Then,  $L^{\r}(X)$ is the minimal closed global attractor if and only if $\Co^{\r}(X)$ is a $T_{2}$  flow.
\end{corollary}

\begin{proof}
Since $X$ is Lagrange stable, we have that for every $x\in X$, $\omega^{\r}(x)\not = \emptyset$. If $M$ is a closed global attractor, then $ \Omega^{r}(X) \subset M$ and  $\overline{\Omega^{r}(X)}\subset \overline{M}=M$. Therefore the unique minimal
global closed attractor of $X$ is $\overline{\Omega^{r}(X)}$. Now it suffices to apply the result given in Theorem \ref{cHausdorff}.
\end{proof}

\begin{theorem} \label{fccomplete} Let $X$ be a locally path-connected, locally compact $T_{2}$ flow satisfying that
$X^{\r}$ is first countable at infinity.   Then,  $L^{\r}(X)= \overline{\Omega^{\r}(X)}$ if and only if $\Co^{\r}(X)$ is a $T_{2}$ $\Co^{\r}$-complete  flow.
\end{theorem}

\begin{proof} It follows from Theorem \ref{cHausdorff} and Theorem \ref{c2complete}.
\end{proof}

\begin{theorem}\label{escapar} Let $X$ be a  locally compact, locally path-connected, connected $T_{2}$ flow and
$\varepsilon(X^{\r})\subset \varepsilon^{\e}(X)$. Then,  $L^{\r}(X)=P(X)=  \overline{\Omega^{\r}(X)}$
if and only if $\Co^{\r}(X)$ is a $T_{4}$ compact  flow.
\end{theorem}

\begin{proof} It is a consequence of Proposition \ref{lem}, Theorem \ref{main} and Theorem  \ref{cHausdorff}.

\end{proof}

\begin{theorem}\label{segcountable} Let $X$ be a second countable,  locally compact, locally  path-connected, connected $T_{2}$ flow and
$\varepsilon(X^{\r})\subset \varepsilon^{\e}(X)$. Then, $L^{\r}(X)=P(X)=  \overline{\Omega^{\r}(X)}$ if and only if
$\Co^{\r}(X)$ is a $T_{4}$ compact $\Co^{\r}$-complete flow.
\end{theorem}

\begin{proof} Suppose  $L^{\r}(X)=P(X)=  \overline{\Omega^{\r}(X)}$. Taking into account that
$X\setminus L^{\r}(X)$ is also second countable and Proposition \ref{lem}, one can find
an increasing sequence $K_n$ of compacts with $X\setminus K_n\in \varepsilon(X^{\r})$ which
covers $X\setminus L^{\r}(X)$. We have that $X\setminus K_n$ is an exterior countable base of $\varepsilon(X^{\r})$.
Now, we just have to use Theorem \ref{escapar} and Theorem  \ref{fccomplete}.
\end{proof}

\begin{corollary} Let $X$ be a  locally  path-connected, connected, compact metric  flow. Then,
  $L^{\r}(X) =\overline{\Omega^{\r}(X)}$ ($L^{\r}(X)$ is the minimal global attractor)  if and only if
$\Co^{\r}(X)$ is a $T_{2}$ compact $\Co^{\r}$-complete flow.
\end{corollary}

\begin{proof} Note that for a compact space, $T_2$ is equivalent to $T_4$, and for  a metric space, second contable is equivalent to Lindel\"of. We also have that a compact space is a Lindel\"of space. Therefore, under these hypothesis
we can apply Theorem \ref{segcountable}.
\end{proof}

\begin{definition} A topological space $X$ is said to be a \emph{Stone space} if $X$ is  Hausdorff, compact and totally disconnected.
\end{definition}

\begin{lemma}\label{ayuda} Let  $X$ be a $T_2$  compact  flow. If $L^{\r}(X)=\overline{\Omega^{\r}(X)}$,
then $\varepsilon(X^{\r})=\{U | L^{\r}(X)\subset U, U \in \t_X\}$.
\end{lemma}
\begin{proof} We always have that $\varepsilon^{\r}(X)\subset \{U | L^{\r}(X)\subset U, U \in \t_X\}$.
In our case, since
$L^{\r}(X)=\overline{\Omega^{\r}(X)}$, by Proposition \ref{lem}, we have that if $K$ is a closed compact
subset and $K\cap L^{\r}(X) = \emptyset$, then $X\setminus K \in \varepsilon^{\r}(X)$. Taking into account that
$X$ is compact, if $L^{\r}(X)\subset U\in \t_X$, then $X\setminus U$ is compact and therefore $U\in \varepsilon^{\r}(X)$.

\end{proof}

\begin{theorem} Let  $X$ be a locally  path-connected, compact metric  flow. Then,
 $X$ is a  $\Co^{\r}$-complete flow  if and only if
 $C(X) =\overline{\Omega^{\r}(X)}$ (i.e. $C(X)$ is the minimal global attractor) and $C(X)$ is a Stone space.
\end{theorem}

\begin{proof}
Suppose that  $X$ is a  $\Co^{\r}$-complete. By  Theorems \ref{cHausdorff} and \ref{propcomplete}, it follows that
$C(X)=L^{\r}(X) =\overline{\Omega^{\r}(X)},$ and by Lemma \ref{ayuda} $\varepsilon(X^{\r})=\{U | L(X)\subset U, U \in \t_X\}$.  Now since $X$ is $T_4$, one has that $X^{\r}$ is locally compact at infinity, that is, for any $E\in \varepsilon(X^{\r})$ there is
$E' \in \varepsilon(X^{\r})$ such that $\overline{E'} \subset E$. Applying Proposition \ref{masmasigual}, we obtain that $C(X)=L^{\r}(X)\cong\F^{\r}(X)$ is a profinite compact space; that is, a Stone space.
The converse follows from Lemma \ref{ayuda} and Theorem \ref{completeex}.

\end{proof}

\begin{remark} We also have a similar version of the results above simply by using the functors $\Co^{\l}, L^{\l}, {\bar L^{\l}}$ and
the notion of repulsor.
\end{remark}

\begin{example} For a Morse function  \cite{Milnor} $f\colon M \to \R$, where
$M$ is a compact $T_2$ Riemannian manifold, one has that the
opposite of the gradient of  $f$ induces a flow with a finite
number of critical points. In this case, we have that $M$ is
locally path-connected and the flow is $\r$-locally compact at infinity. Then we have that
 $L^{\r}(X)=C(X)=L^{\l}(X)$ is a finite set and $X$ is a $\Co^{\r}$-complete and
$\Co^{\l}$-complete flow.
\end{example}

\section{More completion functors for flows}

The authors think that it could be interesting to complete the study presented in this paper
with other posible completions. We suggest to work with $\r$-exterior flows and  the following  functors:

\medskip
a) The functor ${\check W}_{0}
\colon {\bf E^{\r}F} \to {\bf E^{\r}F}$.

Let $(X, \varepsilon(X))$ be an  exterior flow. Notice that we have the following natural transformation
$\omega_{\r}^0 (X)\to \F (X)$.
Taking the  push-out square
$$\xymatrix{L(X) \ar[r]^{e_0} \ar[d] & \omega_{\r}^0 (X)\ar[d]^{\inc_0}\\
X \ar[r]^-{p_0} & X\cup_{L(X)} \omega_{\r}^0 (X)}
$$
and taking similar topologies and externologies  we obtain the completion $ {\check W}_{0}(X)$
and a canonical transformation ${\check W}_{0}(X) \to \Co (X)$.

The authors  think that the study of   this new completion and the possible  relations between
 $\check W_0$-completions  and $\check C_0$-completions will give interesting properties and
 results.

\medskip
b) The functor ${\check C}\colon {\bf E^{\r}F} \to {\bf E^{\r}F}$.

 Given an exterior space $(X,\varepsilon(X))$, for every $E\in \varepsilon(X)$ one can take the
 set of connected componets of $E$ instead of the
 set of path-connected components of $E$. This gives a different notion of end point and
 the corresponding completion functor. Using these new constructions some of the
 above results can be reformulated removing the condition of being first countable at
 infinity.

\medskip
c) The functors $\bCo, \bC \colon {\bf E^{\r}F} \to {\bf E^{\r}F}$.

 In these cases,  for an exterior space $(X,\varepsilon(X))$ and for every $E\in \varepsilon(X)$ one can take the
 set of connected componets of $\overline{E}$ and the
 set of path-connected components of $\overline{E}$. We can take the bar-limit functor  and the correspondings
 end points to construct new completion functors in order to obtain new results about global attractors.

\end{document}